\documentclass{article}

\usepackage{a4wide}
\usepackage[T1]{fontenc}
\usepackage[latin1]{inputenc}

\usepackage{graphicx}
\usepackage{amsmath,amssymb,amsthm}
\usepackage{color}
\usepackage{array}

\newtheorem{theorem}{Theorem}
\newtheorem{lemma}{Lemma}

\newtheorem{corollary}{Corollary}
\newtheorem{prop}{Proposition}
\newtheorem{definition}{Definition}

\theoremstyle{remark}
\newtheorem{remark}{Remark}
\newtheorem{ack}{Acknowledgement\!\!}

\newcommand{\cA}{\mathcal{A}}
\newcommand{\cB}{\mathcal{B}}
\newcommand{\cC}{\mathcal{C}}
\newcommand{\cD}{\mathcal{D}}

\newcommand{\cF}{\mathcal{F}}
\newcommand{\cH}{\mathcal{H}}
\newcommand{\cL}{\mathcal{L}}
\newcommand{\cM}{\mathcal{M}}

\newcommand{\cP}{\mathcal{P}}

\newcommand{\cS}{\mathcal{S}}

\newcommand{\cU}{\mathcal{U}}
\newcommand{\cZ}{\mathcal{Z}}
\newcommand{\nti}{{n\to\infty}}

\newcommand{\BigO}[1] {O\left(#1\right)}
\newcommand{\smallO}[1] {o\left(#1\right)}

\title{On the number of unary-binary tree-like structures with restrictions on the unary height}
\author{Olivier Bodini\thanks{Institut Galil\'ee, Univ. Paris 13, Villetaneuse (France). 
Supported by ANR Metaconc project (France) 2015-18}, 
Dani\`ele Gardy\thanks{DAVID Laboratory, University of Versailles Saint Quentin en Yvelines. Part
of the work of this author was done during a long-term visit at the Institute of Discrete
Mathematics and Geometry of the TU Wien. Partially supported by the P.H.C. Amadeus project {\em
Boolean expressions: compactification, satisfiability and distribution of functions} (2013-14) and
by the ANR project BOOLE (2009-13).}, 
Bernhard Gittenberger\thanks{Institute for Discrete
Mathematics and Geometry, Technische Universit\"at Wien, Wiedner Hauptstrasse 8-10/104,
A-1040 Wien, Austria. 
Supported by FWF grant SFB F50-03 and \"OAD, grant F01/2015}, 
Zbigniew Go{\l}\k{e}biewski\thanks{Institute for Discrete
Mathematics and Geometry, Technische Universit\"at Wien, Wiedner Hauptstrasse 8-10/104,
A-1040 Wien, Austria. 
Supported by FWF grant SFB F50-03}, 
}

\begin{document}
\maketitle

\begin{abstract} 
We consider various classes of Motzkin trees as well as lambda-terms for which we derive
asymptotic enumeration results. These classes are defined through various restrictions
concerning the unary nodes or abstractions, respectively: We either bound their number or the
allowed levels of nesting. 
The enumeration is done by means of a generating function approach and singularity analysis. 
The generating functions are composed of nested square roots and exhibit unexpected phenomena in
some of the cases. Furthermore, we present some observations obtained from generating such terms
randomly and explain why usually powerful tools for random generation, such as Boltzmann samplers,
face serious difficulties in generating lambda-terms. 
\end{abstract}

\section{Introduction}

This paper is mainly devoted to the asymptotic enumeration of lambda-terms belonging to a certain
subclass of the class of all lambda-terms. 
Roughly speaking, a lambda-term is a formal expression built of variables and a quantifier $\lambda$ which
in general occurs more than once and acts on one of the free variables of the subsequent sub-term. 
Lambda calculus is a set of rules for manipulating lambda-terms and
was invented by Church and Kleene in the 30ies (see \cite{Lam1,Lam2,Lam3})
in order to investigate decision problems. It plays an
important r\^ole in computability theory, for automatic theorem proving or as a basis for
some programming languages, e.g. LISP. Due to its flexibility it can be used for a formal
description of programming in general and is therefore an essential tool for analyzing programming
languages (\emph{cf.} \cite{La65a,La65b}) and is now widely used in artificial intelligence.
Furthermore, in typed lambda calculus types can be interpreted as logical propositions and
lambda-terms of a given type as proofs of the corresponding proposition. This is known as the
Curry-Howard isomorphism (see \cite{SU06}) and constitutes in view of the above-mentioned link to
programming a direct relationship between computer programs and mathematical proofs.

Recently, there has been rising interest in random structures related to logic in general 
(see \cite{Za05} \cite{FGGG08}, \cite{GKZ08}, and \cite{DZ09}) and in the properties of random
lambda-terms in particular (see \cite{DGKRTZ10}, \cite{grygiel-lescanne} or \cite{Le13}). 

Although lambda-terms are related to Motzkin trees, the counting sequences of these two objects
have widely different behaviours. In this paper, a tree-like behaviour is meant to be that the
counting sequence asymptotically behaves as is typical for trees with average height
asymptotically proportional to the root of the tree size. 
See \cite{Drm09} for numerous results on such trees as well as many other classes of trees. 
For analyzing the structure of random lambda-terms it is important to know the number of lambda-terms of a given
size. It turns out that this is a very hard problem. The reason is that there are many degrees of freedom for
assigning variables to a given abstraction. This leads to a large number of lambda-terms of fixed
size. If we translate the counting problem into generating functions, then the resulting
generating function has radius 
of convergence equal to zero. Thus none of the classical methods of analytic combinatorics (see
\cite{FlSe09}) is applicable.
Therefore, in this paper we study simpler structures, obtained by bounding either the total number
of abstractions or by introducing bounds on the levels of nesting (either globally or locally, to
be formally defined in the next section) of lambda-terms. 
Note that the number of nesting levels of abstraction or even the number of abstractions in
lambda-terms which occur in computer programming is in general assumed low compared to their size.
E.g., for implementing lambda-calculus we need to bound the height of the underlying stack, which
is determined by the maximal allowed number of nested abstractions. Even more, 
Yang et al.~\cite{YCER11}, who developed the very successful software Csmith for
finding bugs in real programs like the gcc compiler, write on \cite[p.~3]{YCER11}: ``Csmith begins
by randomly creating a collection of struct type declarations. For each, it randomly decides on a
number of members and the type of each member. The type of a member may be a (possibly qualified)
integral type, a bit-field, or a previously generated struct type.'' A declaration in a C
program corresponds to an abstraction in a lambda-term, and the engineers chose the number of
abstractions before randomly generating the rest of the program. That means that they expect the
number of abstractions to be independent of the size of the lambda-term which corresponds to their 
program. Thus, requiring bounds like those mentioned above seems not to be a severe restriction
from a practical point of view. 

Preliminary results on the enumeration of lambda-terms with bounded unary height appeared
in~\cite{BoGaGi11}. 

The plan of the paper is as follows: 
We present all the formal definitions of the objects of our interest in
Section~\ref{sec:description-combinatoire} and then, in Section~\ref{sec:Motzkin}, some results on restricted classes of Motzkin trees for comparison purposes.
The enumeration of lambda-terms with a fixed or bounded number of unary nodes is done in
Section~\ref{sec:nombre-unaires-fixe}. 
Sections~\ref{sec:window} and~\ref{sec:bounded-height} contain the main results of our paper. They
are devoted to the enumeration of
lambda-terms where all bindings have bounded unary length and lambda-terms with bounded unary
height, respectively. In order to achieve our results, we first derive generating functions for
the associated counting problems, which are expressed as a finite nesting of radicals. Then we
study the radii of convergence and the type of their singularities. This will eventually allows us
to determine their number asymptotically, as their size tends to infinity. A comparison of the two
classes of lambda-terms is discussed in Section~\ref{comparison}.
Finally, we investigate how our theoretical results fit with simulations and discover some
challenging facts on the average behaviour of a random lambda-term in
Section~\ref{sec:observations}.

\section{A combinatorial description for lambda-terms}
\label{sec:description-combinatoire}

\subsection{Representation as directed acyclic graphs}
A lambda-term is a formal expression which is described by the context-free grammar
\[ 
T::= a\;\mid\;(T*T)\;\mid\;\lambda a.T
\] 
where $a$ is a variable. The operation $(T*T)$ is called {\em application}.
Using the quantifier $\lambda$ is called {\em abstraction}. 
Furthermore, each abstraction binds a variable and each variable can
be bound by at most one abstraction. A variable which is not bound by an abstraction is called
free. A lambda-term without free variables is called closed, otherwise open.

In this paper we deal with the enumeration of $\alpha$-equivalence classes of closed lambda-terms:
Two terms are $\alpha$-equivalent if one term can be transformed into the other one by a sequence
of $\alpha$-conversions. An $\alpha$-conversion is the renaming of a bound variable in the whole
term ({\em cf.} \cite{Bar:84}). Since the lambda-terms we consider are closed, this means that the
actual variable names are unimportant; only the structure of the bindings is relevant. E.g., we
consider the terms $\lambda x . x$ and $\lambda y . y$ to be identical. 

Furthermore, note that neither application nor iterated abstraction is
commutative, {\em i.e.}, in particular, the terms $\lambda x.\lambda y.T$ and $\lambda y.\lambda
x.T$ are different (if and only if at least one variable $x$ or $y$ appears in
$T$). 

A lambda-term can be represented as an {\em enriched tree,} {\em i.e.}, a graph built from a
rooted tree by adding certain directed edges (pointers). 
First we construct a \emph{Motzkin tree,} {\em i.e.}, a plane rooted tree where each node has
out-degree 0, 1, or 2, if the edges were directed away from the root. 
We denote by the terms {\em leaves}, {\em unary nodes}, and {\em binary nodes}, the nodes with
out-degree 0, 1, and~2, respectively. 
In this tree each application corresponds to a binary node, each abstraction to a unary node, and
each variable to a leaf. 
The fact that an abstraction $\lambda$ binds a variable $v$ is represented by adding a directed
edge from the unary node corresponding to the particular abstraction $\lambda$ to the leaf
labelled by~$v$.
Therefore, each unary node $x$ of the Motzkin tree is carrying (zero, one or more) pointers to
leaves taken from the subtree rooted at $x$; all leaves receiving a pointer from $x$ correspond to
the same variable, and each leaf can receive at most one pointer. The Motzkin tree obtained from a
lambda-term $t$ by removing all pointers (variable bindings) is called the {\em underlying tree}
of $t$.

For instance, the terms $(\lambda x.(x*x)*\lambda y.y)$ and $\lambda y.(\lambda x.x* \lambda x.y)$ correspond to the enriched trees $T_0$ and $T_1$ in Fig.~\ref{treefigure}, 
respectively. In particular, these terms are closed lambda-terms, because every variable is bound by
an abstraction, {\em i.e.}, every leaf receives exactly one pointer.

\begin{figure}[h]
\centering
\setlength{\unitlength}{1cm}
\begin{picture}(7,3.2)
	\put(0,0){\includegraphics[width=7cm,height=3cm]{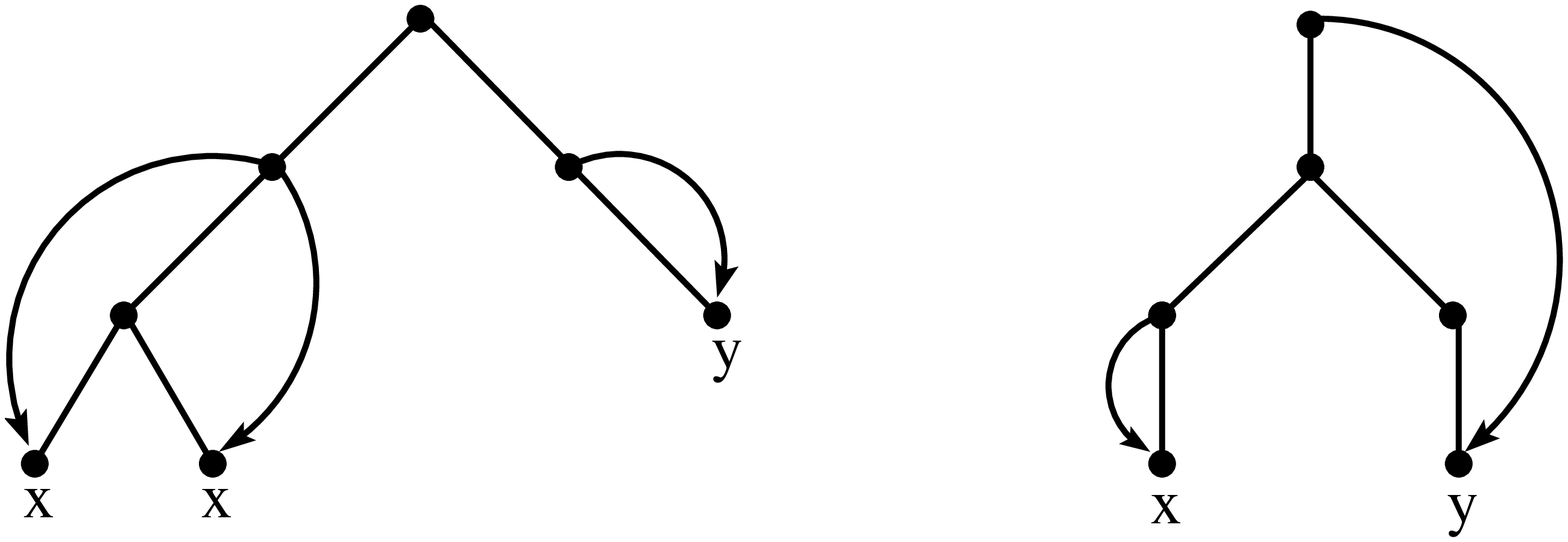}}
	\put(2,0.2){$T_0$}
	\put(5.5,0.2){$T_1$}
\end{picture}
\caption{Two examples of lambda-terms (or $\alpha$-equivalence classes if the labels of the leaves
are removed): Each unary node corresponds to an abstraction $\lambda x$
binding all leaves below it which are labelled by $x$. Binary nodes correspond to applications
merging their  two subtrees $t_1$ and $t_2$ into the more complex structure $t_1*t_2$.}
\label{treefigure}
\end{figure}

As mentioned in the introduction, our interest in the present paper is in lambda-terms with 
restrictions on the number of abstractions and on the number of nesting levels of abstraction,
either locally or globally. The following definitions will allow us to state our restrictions more
precisely. 

\begin{definition}\label{z:pointerlen}
	Consider a lambda-term and its associated enriched tree $T$.
	The {\em unary length} 
	of the binding of a leaf $e$ by some abstraction $v$ in $T$ (directed edge from
	$v$ to $e$), denoted by $l_u(e)$, is defined as a number of unary nodes on the path
	connecting $v$ and $e$ in the underlying Motzkin tree. 
\end{definition}

\begin{definition}
	\label{unaryheight}
	Consider a lambda-term and its associated enriched tree $T$. 
	The \emph{unary height of a vertex} $v$ of $T$, denoted by $h_u(v)$, 
	is defined as number of unary nodes on the path from the root to $v$ in  
	the underlying Motzkin tree. The \emph{unary height of } $T$, $h_u(T)$, is defined by      
	$\max\limits_{\mbox{\scriptsize $v$ \rm vertex of $T$}} h_u(v)$. We use the same notions
	for Motzkin trees as well. 
\end{definition}

In this paper we will enumerate lambda-terms with a fixed number of unary nodes, with bounded unary length of the bindings, or with bounded unary height. 
Of course, other simplifications are possible, such as bounding the number of pointers for each
unary node. Such terms are related to linear (terms where each abstraction binds at most one
variable, also called BCI terms) and affine (terms where each abstraction binds at most one
variables, also called BCK terms) logics as introduced in \cite{Bar:84,II1,II2,KT78}, and their
enumeration was treated in \cite{BGJ13} and generalizations can be found in \cite{BGGJ13} and
\cite{BoGi14}. For their relations to lambda-calculus see for instance \cite{Hi93}.

\subsection{Generating functions associated with lambda-terms}

For each class of lambda-terms we will enumerate the terms of a given size. 
The size of a lambda-term is the number of nodes in the corresponding enriched tree. It is defined
recursively by 
\begin{align*} 
|x|&=1,  \\
|\lambda x.T|&=1+|T|, \\
|(S*T)|&=1+|S|+|T|.
\end{align*} 

In order to count $\alpha$-equivalence classes of lambda-terms of a given size we set up a formal
equation which is then translated into a functional equation for generating functions using the
well-known symbolic method ({\em cf.} \cite{FlSe09}). 

Let us introduce the following atomic classes: the class of
application nodes $\cA$, the class of abstraction nodes $\cU$, the class of free leaves $\cF$, 
and the class of bound leaves $\cD$. Then the class $\cL$ of equivalence classes of lambda-terms
can be described by the specification 
\begin{equation} \label{specifylambda}
\cL=\cF + \left( \cA \times \cL^2 \right) + \left( \cU\times \rm{subs}(\cF\to\cF+\cD,\cL) \right).
\end{equation} 
where the substitution operator $\rm{subs}(\cF\to\cF+\cD,\cL)$ corresponds to replacing some free
leaves in $\cL$ by bound leaves.

\begin{remark}
Note that the lambda-terms specified by $\cL$ are not necessarily closed. Since
$\alpha$-conversion concerns only the bound variables, the equivalence here is w.r.t.
$\alpha$-conversion and substitution of a free variable by another free variable which is not
already present in the term. 
\end{remark}

The specification \eqref{specifylambda} gives rise to a functional equation for the 
bivariate generating function 
$$
L(z,f)=\sum_{t \;\; \mbox{\scriptsize lambda-term}} z^{|t|} f^{\# \mbox{\scriptsize free leaves in } t}
$$
which reads as follows: 
\begin{equation}
L(z,f)=fz+zL(z,f)^2+zL(z,f+1).
\label{eq:def-base-lambdas}
\end{equation}
In particular, the formal generating function for lambda-terms without free variables is 
\begin{align*}
L(z,0)&=[f^0]L(z,f)
\\ &=z^2+2z^3+4z^4+13z^5+42z^6 +139z^7+506z^8+1915z^9+7558z^{10}+\cdots
\end{align*}
Note that these functional equations have to be considered in the framework of formal power
series since the fast growth of the coefficients of the generating function implies that the 
radius of convergence of $L(z,0)$ is zero (see Corollary~\ref{RCglobal} below).

Furthermore note, that the problem of counting closed or open lambda-terms is essentially the same.
Indeed, the formal generating function for open lambda-terms can be derived 
from Eq.~(\ref{eq:def-base-lambdas}) the formula $L(z,1)-L(z,0)=\frac{(1-z)L(z,0)-z L(z,0)^2}{z}.$ Consequently, the problems of
enumerating lambda-terms with or without free variables are of the same difficulty
and the solution for one of them yields the solution for the other one.

Before we start with the analysis of the generating functions associated with the considered
combinatorial structures let us introduce a few further notions. 

\begin{definition}
\label{def:typesingularity}
We say that a function $f:\mathbb C\to \mathbb C$ has a singularity of type $\alpha$ at
$z=\rho$ if there is a constant $c$ such that 
\[
f(z)\sim c\left(1-\frac z\rho\right)^\alpha, 
\]
as $z\to\rho$ inside the domain of analyticity of $f$. 
\end{definition}

\begin{definition}\label{def:dominant_signularity}
If $f:\mathbb C\to \mathbb C$ is a function which is analytic at 0. Then let $S$ denote the set of
all singularities of $f$ which lie on the circle of convergence of the Taylor series of $f$
(expanded at $z=0$). Those singularities in $S$ which are of smallest type are called the dominant
singularities of $f$. 
\end{definition}

\begin{remark}
It is well-known since Darboux \cite{Da78} that the singularities on the circle of convergence
determine the asymptotic behaviour of the coefficients of a series. The transfer theorems of
Flajolet and Odlyzko \cite{FO} make this much more precise and show that indeed only the dominant
singularity in the sense of the definition above and its type yield the (main term of the) asymptotic behaviour. 
\end{remark}

\section{Restricted Motzkin trees}
\label{sec:Motzkin}

Before considering restricted lambda-terms, we present results on classes of restricted Motzkin trees. 
We shall consider classes of Motzkin trees with restrictions analogous to those for lambda-terms,
namely a fixed or bounded number of unary nodes, and a fixed or bounded unary height, where the
unary height of a leaf is the number of unary nodes on the path from the root to that leaf, and
the unary height of a tree is the maximal unary height of a leaf. 

The size of a Motzkin tree is defined as the total number of nodes. The generating function
associated with Motzkin trees satisfies the functional equation
$M(z) = z + z M(z) + z M^2(z).$
Solving this equation shows that the only power series solution is 
\[
M(z) = \frac{1}{2z} \left(  1-z-\sqrt{1-2 z-3 z^2} \right).
\]
The roots of the radicand are $-1$ and $1/3$, the latter being the dominant singularity of $M(z)$
and of type $\frac12$. Applying a transfer theorem from \cite{FO} yields that the number of Motzkin
trees of size~$n$ is asymptotically 
$[ z^ n ] M(z) \sim \frac{3^{n+\frac{1}{2}}}{2 n\sqrt{\pi n}}$. 

\subsection{Restrictions on the total number of unary nodes}

Let $\cM_q$ be the class of Motzkin trees with exactly $q$ unary nodes.
We point out that a Motzkin tree with exactly $q$ unary nodes has a total size $n$ equal to
$q+1+2m$, where $m$ is the number of binary nodes and $m+1$ the number of leaves.

\begin{prop}\label{propMotzkinexactlyq}
The number of Motzkin trees of size $n$ with exactly $q$ unary nodes is 0 if $n \equiv q \mod 2$;  
otherwise it is asymptotically equivalent to 
$\sqrt{\frac{2}{\pi}} \, \frac{1}{2^q \, q!} \, 2^n n^{q-\frac{3}{2}}$, as $n\to\infty$ and for
fixed $q$.
\end{prop}

\begin{proof}
The assertion is an immediate consequence of Tutte's theorem \cite{Tu64} which implies directly 
that the number of Motzkin trees of size $n$ with exactly $q$ unary nodes is
$\frac{(n-1)!}{q!((n-q-1)/2)!}$. 
\end{proof}

For self-containedness and as it is in the flavour of this paper, we offer a proof of
Proposition~\ref{propMotzkinexactlyq} based on analytic combinatorics. 

Obviously $\cM_0 = \cC$ is the class of binary Catalan trees 
and its generating function is $M_0(z) =
C(z) = \frac{1-\sqrt{1-4 z^2}}{2z}$. For $q \geq 1$ we have 
$\cM_q = \cU \times \cM_{q-1} + \sum_{\ell =0}^q \cA \times \cM_\ell \times \cM_{q-\ell}.$
This equation translates into a functional equation for the generating functions and we get (after
solving w.r.t. $M_q(z)$) 
\begin{equation}
M_q (z) = \frac{z M_{q-1}(z) + z \,\sum_{1 \leq \ell \leq q-1} M_\ell(z) M_{q-\ell}(z) }{1-2z M_0(z)}.
\label{eq:Motzkin-unaires-fixes}
\end{equation}
\begin{lemma}\label{lemmaone}
There exists a sequence of polynomials sequence $(P_q(z))_{q \geq 2}$ such that
\begin{equation} \label{assertionMq}
	M_q (z) = \frac{z^{q+1} P_q(z^2)}{(1-4 z^2)^{q-\frac{1}{2}}}, \text{ for } q\ge 2. 
\end{equation} 
The polynomials $P_q(z)$ are given by the recurrence relation 
\begin{equation} \label{recpoly}
P_2(z) = 1; \qquad
P_q(z) =  P_{q-1}(z) + z \sum_{l=2}^{q-2} P_l(z) P_{q-l}(z) \qquad (q \geq 3).
\end{equation} 
\end{lemma}

\begin{proof}
For convenience set $\Delta = 1-2z M_0(z)= \sqrt{1-4 z^2}$. From \eqref{eq:Motzkin-unaires-fixes}
one easily derives that $M_2=z^3/\Delta^3$ which fits with the assertion \eqref{assertionMq}.
We assume that the assumption $M_\ell = \frac{z^{\ell+1} P_\ell(z^2)}{\Delta^{2\ell-1}}$ holds
for $\ell=2,\dots,q-1$; Eq.~(\ref{eq:Motzkin-unaires-fixes}) then gives 
\begin{align*}
M_q(z) &= \frac{z}{\Delta} \; 
\left((1+2 M_1(z))  M_{q-1}(z) + \sum_{\ell=2}^ {q-2} M_\ell(z) M_{q-\ell}(z) \right)
\\ 
&=
\frac{z}{\Delta} \; 
\left( \frac{M_{q-1}(z)}{\Delta} + \sum_{\ell=2}^ {q-2} M_\ell(z) M_{q-\ell}(z) \right)
\\ 
&=
\frac{z}{\Delta} \left(
\frac{z^{q}P_{q-1}(z^2)}{\Delta^{2q-2}} + \sum_{\ell=2}^ {q-2} \frac{z^{\ell+1}
P_\ell(z^2)}{\Delta^{2\ell-1}} \cdot \frac{z^{q-\ell+1} P_{q-\ell}(z^2)}{\Delta^{2(q-\ell)-1}}\right) \\ 
&=
\frac{z^{q+1}}{\Delta^{2q-1}} \;
\left( P_{q-1}(z^2) + z^2\;\sum_{\ell=2}^ {q-2} P_\ell(z^2) P_{q-\ell}(z^2) \right).
\end{align*}
From the last formula we read off the recurrence relation \eqref{recpoly} and get the
assertion after all. 
\end{proof}

The asymptotic behaviour of the coefficients of $M_q(z)$ is now readily obtained (recall that
$n=q+1+2m$, with $m$ being the number of binary nodes):
\[
[z^n] M_q(z) = [ z^{2m} ] \frac{P_q(z^2)}{(1-4 z^2)^{q-\frac{1}{2}}} 
= [ z^{m} ] \frac{P_q(z)}{(1-4 z)^{q-\frac{1}{2}}}
\sim [ z^{m} ] \frac{4^m  P_q(1/4)}{(1-z)^{q-\frac{1}{2}}}.
\]
As $[z^m] (1-z)^{-q+\frac{1}{2}} \sim \frac{m^{q-\frac{3}{2}}}{\Gamma(q-\frac{1}{2})}$, we get  
\[
[z^n] M_q(z) 
\sim  
2^{n-q-1} P_q(1/4) \frac{(n-q-1)^{q-\frac{3}{2}}}{2^{q-\frac{3}{2}} \Gamma(q-\frac{1}{2})}
\sim 
\frac{ P_q(1/4)\sqrt2 }{4^q  \Gamma(q-\frac{1}{2})} \, 2^n  n^{q-\frac{3}{2}}.
\]
Set $a_q = P_q(1/4)$; then $a_2=a_3=1$ and $a_q = a_{q-1}+ \frac{1}{4}
\sum_{2\leq \ell \leq q-2} a_\ell a_{q-\ell}$ for $q \geq 4$.
This implies $a_q=2^{2-q}  C_{q-1}$ where $C_q$ denotes the $q$th Catalan number.
Plugging this into the asymptotic expression for $[z^n] M_q(z)$ gives immediately
Proposition~\ref{propMotzkinexactlyq}.

Next we consider the number of Motzkin trees with {\em at most} $q$ unary nodes. Then we have 
$M_{\leq q}(z) = \sum_{0 \leq r \leq q} M_r(z)$ and $[z^n] M_{\leq q}(z) = \sum_{0 \leq r \leq q}
[z^n] M_r(z)$. 
Hence the last term of the sum gives the asymptotic main term which is $[z^n] M_q$ if $n \not\equiv q \mod 2$, and $[z^n] M_{q-1}$ otherwise.

\subsection{Restrictions on the unary height}

Define $\cB_k$ as the set of Motzkin trees such that all leaves are at the same unary height $k$ 
and $\cB_{\leq k}$ as the set of Motzkin trees where leaves have unary height at most equal to~$k$.

\subsubsection{All leaves at the same unary height}

Again, we start with setting up the specification and translating them into functional equations
for the generating functions. 

\begin{lemma}
The class $\cB_0$ is equal to that of binary Catalan trees. Thus $B_0 (z) = C(z) =
(1-\sqrt{1-4z^2})/(2z)$. 
For $k\ge 1$ we have the recursive specification 
$\cB_k = \cU \times \cB_{k-1} + \cA \times \cB_k \times \cB_k.$ 
Thus, the generating function associated with $\cB_k$ satisfies   
\begin{align*}
B_k (z) &= \frac{1}{2z} \left( 1-\sqrt{1-4 z^2 B_{k-1}} \right)
\\ &=
\frac{1}{2z} \left( 1-\sqrt{1-2z+2z\sqrt{1-2z+\dots + 2z \sqrt{1-4 z^2 }}} \right),
\end{align*}
where the second expression has $k+1$ nested square roots.
\end{lemma}

Now we turn to the asymptotic behaviour of such bounded unary height trees.
For $k=1$, the dominant singularity of $B_1 (z)$ is at $z=1/2$ and of type $\frac{1}{4}$. The
other singularity is at $z=-1/2$, but of type $\frac{1}{2}$ and gives therefore an
asymptotically negligible contribution. We obtain 
\[
[z^n] B_1(z) \sim \frac{1}{4} \cdot \frac{2^{\frac{1}{4}} 2^n n^{- \frac{5}{4}}}{\Gamma \left( \frac{3}{4} \right)}.
\]
Likewise, for $k \geq 2$, the singularities of $B_k(z)$ are $\pm 1/2$, which can easily be seen by
induction. The singularity at $z=-1/2$ originates from the innermost radical only and is therefore
of type $1/2$. At $z=1/2$ all radicals vanish at once and hence the singularity is of type
$1/2^{k+1}$. Consequently, as $z\to 1/2$, we have  
\[
	B_k(z) = 1 - 2^{2^{-k-1}} (1-2z)^{2^{-k-1}} (1+O(\sqrt{1-2z}\,)).
\]
Determining the asymptotic behaviour is now straightforward.

\begin{prop}
The number of Motzkin trees in which all leaves have exactly unary height $k$ is 
\[
[z^n] B_k(z) \sim  2^{n+\alpha} \, \frac{n^{-1-\alpha}}{\Gamma(-\alpha)} 
\qquad
with \qquad\alpha = \frac{1}{2^{k+1}}.
\]
\end{prop}

\begin{remark}
This is another of the rather rare examples where the generating function of a recursively
specified combinatorial structure does not have a dominant singularity of type 1/2 (or multiple of
1/2). A general discussion of possible singularity types of generating functions given by systems
of functional equations was recently given by Banderier and Drmota~\cite{BD15}.
\end{remark}

\subsubsection{Motzkin trees of bounded unary height}

The case $k=0$ again corresponds to binary Catalan trees and for larger $k$ a similar recursive
specification as in the previous subsection holds. 

\begin{lemma}
The class $\cB_{\le k}$ is equal to that of binary Catalan trees. For $k\ge 1$ we have the
recursive specification 
$\cB_{\leq k} = \cZ + \cU \times \cB_{\leq k-1} + \cA \times \cB_{\leq k} \times \cB_{\leq k}$.
Thus the generating function associated with $\cB_{\le k}$ satisfies
\begin{align*}
B_{\leq k}(z) &= \frac{1}{2z} \left( 1-\sqrt{1-4 z^2 - 4 z^2 B_{\leq k-1}} \right)
\\ &=
\frac{1}{2z} \left( 1-\sqrt{1-2z-4 z^2+2z\sqrt{1-2z-4z^2 +\dots + 2z \sqrt{1-4 z^2 }}} \right).
\end{align*}
where the second expression has $k+1$ nested square roots.
\end{lemma}

Again, the first function $B_{\leq 0}(z)$ has the two singularities $\pm 1/2$, but the next ones
have different singularities. Indeed, the innermost square root $\sqrt{1-4z^2}$ has a zero at
$z=\pm 1/2$, but the next radical, $\sqrt{1-2z-4z^2 + 2z \sqrt{1-4z^2}}$, has a zero at
$\tilde{\rho}_1 = 0.4064073933<1/2$. The following few values are $\tilde{\rho}_2 = 0.3759923651$,
$\tilde{\rho}_3 = 0.3617581845$,  $\tilde{\rho}_4 = 0.3538076738$.

\begin{lemma}
\label{lemma:zeros-motzkin-hauteur-bornee}
 Let $\tilde{R}_0(z) = 1-4z^2$ and $\tilde{R}_k(z) = 1-2z-4z^2 + 2z \sqrt{\tilde{R}_{k-1}(z)}$.
 Then the values $\tilde{\rho}_k$, defined as the smallest real positive root of $\tilde{R}_k(z) =
 0$, form a decreasing sequence.
\end{lemma}

\begin{proof}
An easy inductive argument shows that the functions $\tilde{R}_k(z)$ are decreasing functions on
the positive real line (of course, only up to their first singularity) 
and smaller than~1 there: Note that for positive $z$ we get
$\sqrt{\tilde{R}_{k-1}(z)} \le 1 \leq 1+4z$. 

Notice that the class of Motzkin trees of bounded unary height is a subclass of the class of
unrestricted Motzkin trees. The generating function of the latter one has dominant singularity
equals to $\frac13$. Hence, for any fixed $k$, we must have $\tilde{\rho}_k \geq \frac13$.

Now, suppose that $\tilde{\rho}_{k-1} \leq \tilde{\rho}_{k}$. Then, since $\tilde{R}_{k}(z)$ is
decreasing for positive real $z$ and $\tilde{R}_{k-1}(\tilde{\rho}_{k-1}) = 0$, we have
$\tilde{R}_{k}(\tilde{\rho}_{k-1}) = 1 - 2 \tilde{\rho}_{k-1} - 4 \tilde{\rho}_{k-1}^2 \ge 0$. But
$1 - 2 z - 4 z^2 \ge 0$ if and only if $z\in \left[\frac{-\sqrt{5}-1}{4},
\frac{\sqrt{5}-1}{4}\right]$ and $\frac{\sqrt{5}-1}{4} < 1/3$ which contradicts the fact that 
$\tilde{\rho}_k \geq \frac13$ for all $k$.
\end{proof}

\begin{remark}
Since the sequence $(\tilde{\rho}_k)_{k\ge 0}$ is decreasing and bounded from below by $\frac13$, 
one can
try to prove that $\tilde{\rho}_k \to \frac13$ as $k \to \infty$. Though numerical evidence
supports this, it seems not obvious at all. Since it is not the key point of our paper we decided
to skip it.
\end{remark}

As $B_{\leq k}(z) = \left(1-\sqrt{\tilde{R}_k(z)}\right) /(2z) $ 
and each radical has a different dominant
singularity, the dominant singularity of $B_{\leq k}$ is at $z=\tilde{\rho}_k$ and of type $1/2$.
Here the dominant singularity always comes from the outermost radical. 
Thus, we obtain the following result:

\begin{prop}
The number of Motzkin trees with unary height at most equal to~$k$ is
\[
[z^n] B_{\leq k}(z) \sim \tilde{C} \tilde{\rho}_k^{-n} n^{-\frac32}
\]
where $\tilde{\rho}_k$ is defined in Lemma~\ref{lemma:zeros-motzkin-hauteur-bornee} and $\tilde{C}$ is a suitable constant.
\end{prop}

\section{Enumeration of lambda-terms with prescribed number of unary nodes}
\label{sec:nombre-unaires-fixe}

\subsection{Recurrence for the generating functions}
We consider here the set $\cS_q$ of lambda-terms that have {\em exactly} $q$ unary nodes. As a
consequence their unary height is obviously bounded.
We shall set up recurrence relations for the generating functions~$S_q$.
Let $z$ mark the total size and $f$ mark the number of free leaves.
The objects in $\cS_0$ are again binary Catalan trees and all the leaves are free (since there is
no unary node). Thus 
\[
S_0 (z,f) = \frac{1 - \sqrt{1-4 f z^2 }}{2z}. 
\]
For $q=1$ either the unique unary node is equal to the root -- each leaf of the whole tree then
either becomes bound or stays free -- or the root is a binary node and the unique unary node
appears either in the left or the right subtree. This yields the specification 
\[
\cS_1 =  \left( \cU\times subs(\cF\to\cF+\cD, \cS_0 ) \right)  + (\cA, \cS_0, \cS_1) + (\cA, \cS_1, \cS_0).
\]
and a recurrence relation for the generating function: 
\begin{equation}
S_1 (z,f) = z S_0 (z,f+1) + 2z S_0(z,f) \; S_1(z,f).
\label{eq:nbre-unaires-1}
\end{equation}
Solving, we get
\[
S_1(z,f) = \frac{z S_0 (z,f+1)}{1-2z S_0(z,f)}
= \frac{1-\sqrt{1-4(f+1)z^2}}{2 \sqrt{1-4fz^2}}.
\]
For general $q \geq 1$ a term  has either a unary node as root and $q-1$ unary nodes below or a
binary node as root, and the $q$ unary nodes are split into $\ell$ nodes assigned to the left
subtree, and $q-\ell$ nodes assigned to the right subtree. Hence we obtain 
\[
\cS_q = \left( \cU \times subs(\cF\to\cF+\cD, \cS_{q-1} ) \right) 
+ \sum_{\ell = 0}^q \; (\cA, \cS_{\ell} , \cS_{q-\ell} ) ,
\]
which gives
\[
S_q(z,f) = z S_{q-1}(z,f+1) + z \sum_{\ell=0}^q S_l(z,f) \; S_{q-l}(z,f).
\]
We can easily solve it and obtain $S_q(z,f)$ in terms of the $S_\ell(z,f)$ for $\ell < q$:
\begin{equation}
S_q (z,f) = \frac{z}{1-2z S_0(z,f)} 
\left(
S_{q-1}(z,f+1) + \sum_{\ell=1}^{q-1} S_\ell(z,f) \;  S_{q-\ell}(z,f).
\right)
\label{eq:rec-nbre-unaires-fixe}
\end{equation}
The number of closed lambda-terms, which we are interested in, is then 
\begin{equation} \label{gfclosedterms}
S_q (z,0) =
S_{q-1}(z,1) + \sum_{\ell=1}^{q-1} S_l(z,0) \;  S_{q-l}(z,0).
\end{equation} 

\subsection{Solving the recurrence}
\begin{lemma}
\label{lemma:fg-nombre-unaires-fixe}
Let $\sigma_q(f) = \sqrt{1-4(f+q)z^ 2}$ for $q \geq 0$.\footnote{$\sigma_q$ is actually a function
in the two variables $z$ and $f$, but $z$ plays no r\^ole in the statement and proof of this Lemma.}
Then, for all $q \geq 0$, there exists a rational function $R_q$ in $q+1$ variables such that
\begin{equation}\label{eq:l2}
S_q(z,f) = - \frac {z^{q-1} \sigma_q(f)} {2 \, \prod_{\ell=0}^{q-1} \sigma_\ell(f)} + R_q(z,
\sigma_0(f), \dots, \sigma_{q-1}(f)).
\end{equation}

Moreover, the denominator of $R_q(z, \sigma_0(f), \dots, \sigma_{q-1}(f))$ is of the form 
$\prod_{0 \leq \ell < q} \sigma_\ell(f)^{\alpha_{\ell,q}} $ where the exponents
$\alpha_{0,q},\dots,\alpha_{q-1,q}$ are positive integers. 
\end{lemma}
\begin{proof}
The proof is based on induction on~$q$. To start the induction observe that $S_0 (z,f) =
\frac{1-\sigma_0(z,f)}{2z}$ and $R_0=0$. Now assume that \eqref{eq:l2} is true for $S_0 (z,f),
\dots, S_q(z,f)$. Then by \eqref{eq:rec-nbre-unaires-fixe} and $\sigma_0(f) = 1-2z S_0(z,f)$ 
we have 
\begin{align*} 
S_{q+1}(z,f) = & \frac{z}{\sigma_0(f)} 
\left(
- \frac {z^{q-1} \sigma_q(f+1)} {2 \, \prod_{\ell=0}^{q-1} \sigma_\ell(f+1)} \right. \\
&\qquad \left.+ R_q(z,
\sigma_0(f+1), \dots, \sigma_{q-1}(f+1))
+ \sum_{\ell=1}^{q-1} S_\ell(z,f) S_{q-\ell}(z,f)
\right).
\end{align*} 
By observing that $\sigma_q(f+1) = \sigma_{q+1}(f)$ we obtain
\[
S_{q+1}(z,f) = \frac{z}{\sigma_0(f)} 
\left(
- \frac {z^{q-1} \sigma_{q+1}(f)} {2 \, \prod_{\ell=0}^{q-1} \sigma_{\ell+1}(f)} + R_q(z,
\sigma_1(f), \dots, \sigma_{q}(f))
+ \sum_{\ell=1}^{q-1} S_\ell(z,f) S_{q-\ell}(z,f)
\right).
\]
The induction hypothesis implies that each $S_\ell(z,f)$ is itself a rational function of $z$,
$\sigma_0(f)$, $\sigma_1(f)$, \dots, $\sigma_\ell(f)$. Hence, by setting
$R_{q+1} = (z/ \sigma_0(f))  R_q (z, \sigma_1(f), \dots, \sigma_{q}(f))  + \sum_{\ell=1}^{q-1} S_\ell(z,f) S_{q-\ell}(z,f)$ we obtain 
\[
S_{q+1}(z,f) = - \frac {z^{q} \sigma_{q+1}(f)} {2 \, \prod_{\ell=0}^{q} \sigma_\ell(f)} + R_{q+1}(z, \sigma_0(f), \dots, \sigma_{q}(f)).
\]
The expression of the denominator of the $R_q$ comes readily from the recurrence expression.
\end{proof}

By setting $f=0$, we obtain the following lemma: 
\begin{lemma}
The generating function enumerating all closed terms with exactly $q$ unary nodes is
\begin{equation}
S_q(z,0) = - \frac {z^{q-1} \sqrt{1-4qz^2}} {2 \, \prod_{\ell=0}^{q-1} \sqrt{1-4\ell z^2}} + R_q(z, 1, \sqrt{1-4 z^2},\dots,\sqrt{1-4(q-1)z^2}),
\label{eq:fg-close-unaires-fixes}
\end{equation}
where the rational function $R_q$ comes from Lemma~\ref{lemma:fg-nombre-unaires-fixe}.
Its dominant singularities are $z=\pm \frac{1}{2\sqrt{q}}$.
\end{lemma}

 \subsection{Asymptotics}
A lambda-term with exactly $q$ unary nodes and $i$ leaves has $i-1$ binary nodes and size
$n=q+2i-1$.
From Lemma~\ref{lemma:fg-nombre-unaires-fixe}, the term $R_q$ will have singularities at $z = \pm
1/(2\sqrt{\ell})$ for $1\leq \ell < q$.
The first term in the right-hand side of \eqref{eq:fg-close-unaires-fixes} has
singularities of smaller type at $z = 1/(2\sqrt{q})$ than the second term. Hence it gives the
dominant contribution to the asymptotics of $[z^n]S_q(z,0)$: 
\[
[ z^n ] S_q(z,0) \sim [ z^{q+2i-1} ] \frac{- z^{q-1} \sqrt{1-4q z^2}}{2 \prod_{\ell = 1}^{q-1} \sqrt{1-4\ell z^2}}
\sim [ z^{i} ] \frac{- \sqrt{1-4q z}}{2 \prod_{\ell = 1}^{q-1} \sqrt{1-4\ell z}}, \text{ as }\nti.
\]
The denominator $\prod_{\ell = 1}^{q-1} \sqrt{1-4\ell z}$ contributes a multiplicative factor
$$ 
\prod_{\ell = 1}^{q-1} \sqrt{1-\frac\ell q} =q^{(1- q)/2}   \sqrt{ (q-1)! } 
$$ 
and we obtain:
\begin{prop}
\label{prop:nbre-unaires-fixe}
The number of closed lambda-terms with exactly $q$ unary nodes and size $n$ is 0 if $n=q \; mod \; 2$; otherwise its asymptotic value is
\[
[ z^n ] S_q(z,0)
\sim \frac {\sqrt2}{2^q \sqrt{(q-1)!}  \, \sqrt{\pi n^3}} \, (2\sqrt{q})^n, \text{ as }\nti.
\]
\end{prop}

\begin{remark} 
Though \eqref{eq:Motzkin-unaires-fixes} and \eqref{eq:rec-nbre-unaires-fixe} have a very similar shape, the
results of Propositions~\ref{propMotzkinexactlyq} and~\ref{prop:nbre-unaires-fixe} 
are rather different. But note that
even though \eqref{eq:rec-nbre-unaires-fixe} was the starting point, we eventually use
\eqref{gfclosedterms} instead. Thus the resonance-like behaviour induced by
\eqref{eq:Motzkin-unaires-fixes}
and leading to the singularity of lower-order type described in Lemma~\ref{lemmaone} disappears. 
\end{remark}

\subsection{Lambda-terms with at most $q$ unary nodes}

We denote by $S_{\leq q} (z,f)$ the generating function for lambda-terms with {\em at most} $q$
unary nodes, where again $z$ marks the nodes, and $f$ the free leaves.
If $q=0$ we get once more the generating function for binary Catalan trees: $S_{\leq 0}(z) =
S_0(z)=C(z)$. 
Otherwise, $S_{\leq q}(z) = \sum_{\ell=0}^q S_\ell(z)$ and hence we can apply the results we
obtained for a fixed number of unary nodes. 
The dominant singularity of $S_{\leq q}(z)$ comes from $S_q(z)$, whereas the terms $S_\ell(z)$ for
$\ell<q$ give negligible contributions to the asymptotics:
The terms with exactly $q$ unary nodes outnumber those with at most $q-1$ such nodes and
determine the asymptotic behaviour of the number of terms, which is the same for a fixed or
bounded number of unary nodes.

\section{Enumeration of lambda-terms with bounded unary length of bindings}\label{sec:window}

Now we turn our attention to the problem of enumerating lambda-terms with bounded unary length of
their bindings (for the definition see Def.~\ref{z:pointerlen}).

Let $\mathcal{G}_{\leq k}$ denote the class of closed lambda-terms where all bindings have unary
length less than or equal to $k$. Our goal is to set up an equation specifying 
$\mathcal{G}_{\leq k}$. 

Define $\hat{\mathcal{P}}^{(i,k)}$ as the class of unary-binary trees such that every leaf $e$ can
be labelled in $\min\{h_{u}(e)+i,k\}$ ways.
The classes $\hat{\mathcal{P}}^{(i,k)}$ can be recursively specified, starting from a class
$\mathcal{Z}$ of atoms, by
\[
\hat{\mathcal{P}}^{(k,k)} = k \mathcal{Z} + (\mathcal{A} \times \hat{\mathcal{P}}^{(k,k)} \times \hat{\mathcal{P}}^{(k,k)}) + (\mathcal{U} \times \hat{\mathcal{P}}^{(k,k)})
\]
and
\[
\hat{\mathcal{P}}^{(i,k)} = i \mathcal{Z} + (\mathcal{A} \times \hat{\mathcal{P}}^{(i,k)} \times
\hat{\mathcal{P}}^{(i,k)}) + (\mathcal{U} \times \hat{\mathcal{P}}^{(i+1,k)}), 
\]
for $i < k$. 
Using again the traditional correspondence between specifications and generating functions we
obtain 
\begin{equation}\label{z:eq:2}
\hat{P}^{(k,k)}(z) = \frac{1 - z - \sqrt{(1-z)^2 - 4 k z^2}}{2 z}
\end{equation}
and
\begin{equation}\label{z:eq:1}
\hat{P}^{(i,k)}(z) = \frac{1 - \sqrt{1 - 4 i z^2 - 4 z^2 \hat{P}^{(i+1,k)}(z)}}{2 z}, 
\end{equation}
for $i < k$. 

Note that for every positive integer $k$, the class $\hat{\mathcal{P}}^{(k,k)}$ consists of all
Motzkin trees with $k$ types of leaves. Moreover, the class $\hat{\mathcal{P}}^{(0,k)}$ is
isomorphic to the class $\mathcal{G}_{\leq k}$ and thus the recursive specification gives directly
the generating function $G_{\leq k}(z) = \hat{P}^{(0,k)}(z)$ associated with $\mathcal{G}_{\leq
k}$.

We can rewrite~\eqref{z:eq:1} and~\eqref{z:eq:2} in the form 
\begin{equation}\label{z:eq:21}
\hat{P}^{(i,k)}(z) = \frac{1}{2z} \left(1 - \mathbf{1}_{[i=k]} z - 
\sqrt{\hat{R}_{k-i+1,k}(z)}\right),
\end{equation}
where
\begin{align*}
\hat{R}_{1,k}(z) &= (1-z)^2 - 4kz^2,\\
\hat{R}_{2,k}(z) &= 1 - 4 (k-1) z^2 - 2z + 2z^2 +2z\sqrt{\hat{R}_{1,k}(z)},
\end{align*}
and 
\begin{align}
\hat{R}_{i,k}(z) &= 1 - 4 (k-i+1) z^2 - 2z + 2z\sqrt{\hat{R}_{i-1,k}(z)}, 
\label{R_k_nesting}
\end{align}
for $3 \leq i \leq k+1$. Hence, $G_{\leq k}(z) = \frac{1-\sqrt{\hat{R}_{k+1,k}(z)}}{2z}$.

\subsection{Analysis of the radicands}

Let us now introduce the definition of a dominant radicand.

\begin{definition}\label{def:dom_rad}
	Consider a function $f(z)$ which is analytic at $z=0$, but not entire, and of the form  
\[
	f(z) = \frac{1 - \sqrt{p_k(z) + q_k(z)\sqrt{p_{k-1}(z) + q_{k-1}(z) \sqrt{\ldots
	\sqrt{p_1(z)}}}}}{2z}
\]
where $p_i(z)$ ($i=1,\dots,k$)  and $q_i(z)$ ($i=2,\dots,k$) are polynomials in $z$. 
We call its $j$-th radicand, which is $p_1(z)$ if $j=1$ and $p_j(z) + q_j(z)\sqrt{\ldots}$
otherwise, a dominant radicand if it has a zero at a dominant singularity of $f(z)$. 
\end{definition}

In order to proceed, we need to know the location and type of the dominant singularity
$\hat{\rho}$ of the ``global'' generating function $G_{\leq k}(z)$. 
This means actually that we need to know which radicands are dominant. 

Nested structures appear frequently in combinatorial objects. Often these structures lead to
generating functions of the form of continued fractions (see for example  \cite{flajolet,bouttier}).
Nested radicals are less frequent. They occur for example when enumerating binary non-plane trees
\cite{otter,FlSe09,flajolet-bona}, where there appears an ``iterated square-root'' expansion.

\begin{lemma}\label{z:l:2}
For every $k>0$ and $1 \leq j \leq k+1$, the function $\hat{R}_{j,k}(z)$ is strictly 
decreasing on the
positive real line (in the interval where it is defined as a real-valued function).
\end{lemma}
\begin{proof}
We proceed by induction on $j$: $\hat{R}_{1,k}(z)=(1-z)^2 - 4 k z^2$ is clearly decreasing for $z$ real positive and $k > 0$. 
Now assume $\hat{R}_{j-1,k}(z)$ is decreasing for $z>0$. Thus, for positive $z$ we have  
\begin{align*}
\frac {\rm d}{{\rm d}z} \hat{R}_{j,k}(z) &= -8 \left( k-j+1 \right) z-2
+2\sqrt {\hat{R}_{j-1,k}(z)}+{\frac {z
{\frac {\rm d}{{\rm d}z}} \hat{R}_{j-1,k}(z)}{\sqrt {\hat{R}_{j-1,k}(z)}}}.
\end{align*}
The induction hypothesis implies $\hat{R}_{j-1,k}(z) \leq \hat{R}_{j-1,k}(0)=1$ and $\frac {\rm
d}{{\rm d}z} \hat{R}_{j-1,k}(z) < 0$, which eventually gives $\frac {\rm d}{{\rm d}z}
\hat{R}_{j,k}(z)< - 8 (k-j+1) z \leq 0$ for real positive~$z$. 
\end{proof}

Observe that the function $\sqrt{\hat{R}_{k,k+1}}$ has the same dominant singularity as the
function $G_{\leq k}(z)$.

\begin{lemma}\label{nocomplexroots}
Assume $k>0$ and that the radical $\sqrt{\hat{R}_{j,k}(z)}$ has a positive singularity and let $z_0$ denote
the smallest one. Then there are no complex singularities having the same modulus as $z_0$.
\end{lemma}
\begin{proof}
From Eq.~(\ref{z:eq:21}) we know that $\hat{R}_{j,k}(z) = (1 - \mathbf{1}_{[j=k]} z - 2 z
\hat{P}^{(k-j+1,k)}(z))^2$. First, assume that 
$z_0$ is a root of $\hat{R}_{j,k}(z)$. Then $2 z_0
\hat{P}^{(k-j+1,k)}(z_0) + \mathbf{1}_{j=k} z_0 = 1 $.
If there were another (complex) root $x = z_0 e^{i \theta}$ of the same modulus, 
then we would have
\[
1 = 2 z_0 \hat{P}^{(k-j+1,k)}(z_0) + \mathbf{1}_{[j=k]} z_0 = \left| 2 z_0 e^{i \theta}  
P^{(k-j+1,k)}(z_0 e^{i \theta}) + \mathbf{1}_{j=k} z_0 e^{i \theta} \right|.
\]
Since $\hat{P}^{(k-j+1,k)}(z) = \sum_{n} \hat{a}_{j,k,n} z^n$ can be viewed as the generating
function of some suitable class of lambda terms, for all sufficiently large $n$ we have 
$\hat{a}_{j,k,n} > 0$. But this implies that  
\[
\left| 2 z_0 e^{i \theta}  P^{(k-j+1,k)}(z_0 e^{i \theta}) + \mathbf{1}_{j=k} z_0 e^{i \theta} \right| < 1
\]
whenever $\theta \neq 0$, which leads to a contradiction.

If $z_0$ is not a root of $\hat{R}_{j,k}(z)$, then $z_0$ must be a zero of some
$\hat{R}_{j-\ell,k}(z)$ with suitable $\ell>0$. This follows from the nested structure
\eqref{R_k_nesting} of the radicals. But then we can apply the arguments above to
$\hat{R}_{j-\ell,k}(z)$ and arrive again at a contradiction.  
\end{proof}

Let us now study the exact location and type of the dominant singularity of the functions $G_{\leq
k}(z)$. The next lemma will also prove that the singularity in the assumption of the previous
lemma indeed exists.

\begin{lemma}\label{z:l:1}
Let $\hat{\rho}_k$ be the dominant singularity of the function $G_{\leq k}(z)$. Then $\hat{\rho}_{k} = \frac{1}{1 + 2 \sqrt{k}}$ comes from the innermost radicand and is of type $\frac12$.
\end{lemma}

\begin{proof}
If a positive root of the radicand $\hat{R}_{i,k}(z)$ exists, denote its smallest one as 
$\hat\rho_{i,k}$.
Let us consider the roots of the innermost radicand $\hat{R}_{1,k}(z)$.
Since $\hat{R}_{1,k}(z)$ is a quadratic equation, we know that it has two roots: 
$\frac{1}{1+2\sqrt{k}}$ and $\frac{1}{1-2\sqrt{k}}$.
Moreover, since $k$ is a positive integer, $\hat\rho_{1,k} = \frac{1}{1+2\sqrt{k}}$ is the
dominant singularity of the generating function $\hat{P}^{(1,k)}(z)$ and of type $\frac12$. 

Let us now prove that none of the radicands $\hat{R}_{j,k}(z)$, $2 \leq j \leq k+1$, has a positive
root which is smaller than or equal to $\hat\rho_{1,k}$.
By induction on $j$, using the formula $\hat\rho_{1,k} = \frac{1}{1+2\sqrt{k}}$, and simplifying, we
obtain $\hat{R}_{2,k}(\hat\rho_{1,k}) = 5 \hat\rho_{1,k}^2 > 0$. Furthermore, from
Lemma~\ref{z:l:2} we know that $\hat{R}_{2,k}(z)$ is decreasing on $\mathbb{R}_+$. Hence,
$\hat{R}_{2,k}(z)$ does not have any positive root not larger than $\hat\rho_{1,k}$.
Assume that $\hat{R}_{j,k}(z)$ (for some $j \geq 2$) does not have any positive root smaller than
or equal to $\hat\rho_{1,k}$. Then we get $\hat{R}_{j+1,k}(\hat\rho_{1,k}) = (4j-1)
\hat\rho_{1,k}^2 +2 \hat\rho_{1,k} \sqrt{\hat{R}_{j,k}(\hat\rho_{1,k})} > 0$ and again using the
argument that $\hat{R}_{j+1,k}(z)$ is decreasing on the positive real line, we obtain that
$\hat\rho_{1,k}$ is the dominant singularity of $\hat{R}_{j+1,k}(z)$ and of type $\frac12$. 

Thus, $\hat\rho_{1,k}$ is a dominant singularity of $G_{\leq k}(z)$, and
Lemma~\ref{nocomplexroots} implies that it is the only one.
\end{proof}

The following proposition will be useful to derive the asymptotic behaviour of the number of
lambda-terms in the considered class of terms.

\begin{prop}\label{z:pr:1}
Let $\hat\rho_k$ be the root of the innermost radicand $\hat{R}_{1,k}(z)$. Then
\begin{equation} \label{firstclaim}
\hat{R}_{1,k}(\hat\rho_k (1-\epsilon)) = 2 (1 - \hat\rho_k) \epsilon + \BigO{\epsilon^2}
\end{equation} 
and 
\[
\hat{R}_{j,k}(\hat\rho_k (1-\epsilon)) = c_j \hat\rho_{k}^2 + \frac{4 \hat\rho_{k}^{\frac32}
k^{\frac14}}{\sqrt{\prod_{l=2}^{j-1} c_l}} \sqrt{\epsilon} + \BigO{\epsilon},
\]
for $2 \leq j \leq k+1$, where
$c_{1} = 1$ and $c_j = 4 j - 5 + 2 \sqrt{c_{j-1}}$ for $2 \leq j \leq k+1$.
\end{prop}
\begin{proof}
Using the Taylor expansion of $\hat{R}_{1,k}(z)$ around $\hat\rho_k$ we obtain 
\[
\hat{R}_{1,k}(z) = \hat{R}_{1,k}(\hat\rho_k) + (z - \hat\rho_k) \frac{d}{dz}
\hat{R}_{1,k}(\hat\rho_k) + \BigO{(z-\hat\rho_{k})^2}.
\]
Knowing that $\hat{R}_{1,k}(z)$ has a zero at $z = \hat\rho_{k}$ and setting $z = \hat\rho_{k}
(1-\epsilon)$ we obtain the first claim \eqref{firstclaim}.

The next step is computing an expansion of $\hat{R}_{j,k}(z)$ around $\hat\rho_k$, where $2 \leq j
\leq k+1$. From \eqref{firstclaim} we conclude that
\[
\sqrt{\hat{R}_{1,k}(\hat\rho_k (1-\epsilon))} = \sqrt{2 (1 - \hat\rho_k)} \sqrt{\epsilon} + \BigO{\epsilon}
\]
and from the recursive relation \eqref{R_k_nesting} for $\hat{R}_{j,k}(z)$ we have
\[
\hat{R}_{2,k}(\hat\rho_k (1-\epsilon)) = 1 - 2 \hat\rho_k + 6 \hat\rho_k^2 - 4 k \hat\rho_k^2 + 2 \hat\rho_k \sqrt{2 (1 - \hat\rho_k)} \sqrt{\epsilon} + \BigO{\epsilon}.
\]
Using the formula $\hat\rho_k = \frac{1}{1+2\sqrt{k}}$ and simplifying we get
\[
\hat{R}_{2,k}(\hat\rho_k (1-\epsilon)) = 5 \hat\rho_k^2 + 4 \hat\rho_k^{\frac32} k^{\frac14} \sqrt{\epsilon} + \BigO{\epsilon}.
\]
Assume that for $2 \leq j \leq k+1$ we have $\hat{R}_{j,k} = c_j \hat\rho_k^2 + d_j \sqrt{\epsilon} + \BigO{\epsilon}$.
We just checked that this holds for $j=2$ with $c_2 = 4 \cdot 2 - 5 + 2 \sqrt{1} = 5$ and $d_2 = 4
\hat\rho_k^{\frac32} k^{\frac14}$. Now, we proceed by induction: Observe that
\[
\hat{R}_{j+1}(\hat\rho_k (1 - \epsilon)) = 1 - 4(k-j) \hat\rho_k^2 (1 - \epsilon)^2 - 2 \hat\rho_k
(1 - \epsilon) + 2 \hat\rho_k (1 - \epsilon) \sqrt{c_j \hat\rho_k^2 + d_j \sqrt{\epsilon} +
\BigO{\epsilon}}.
\]
Expanding, using the formula $\hat\rho_k = \frac{1}{1+2\sqrt{k}}$, and simplifying we obtain
\[
\hat{R}_{j+1}(\hat\rho_k (1 - \epsilon)) = (4j-1+2\sqrt{c_j}) \hat\rho_k^2 + \frac{d_j}{\sqrt{c_j}} \sqrt{\epsilon} + \BigO{\epsilon}.
\]
Setting 
$c_{j+1} = 4j-1+2\sqrt{c_j}$ and $d_{j+1} = \frac{d_j}{\sqrt{c_j}}$ 
for $2 \leq j \leq k$, we obtain $\hat{R}_{j+1}(\hat\rho_k (1 - \epsilon)) = c_{j+1} \hat\rho_k^2
+ d_{j+1} \sqrt{\epsilon} + \BigO{\epsilon}.$
Expanding $d_{j+1}$ using its recursive relation and $d_2 = 4 \hat\rho_k^{\frac32} k^{\frac14}$ we
have for $2 \leq j \leq k$ 
\[
d_{j+1} = \frac{4 \hat\rho_k^{\frac32} k^{\frac14}}{\prod_{l=2}^{j} \sqrt{c_l}}.
\qedhere
\]
\end{proof}

We are now in the position to give the asymptotic behaviour of the number of lambda-terms having
only bindings of bounded unary length.

\begin{theorem}\label{z:thm1}
Let for any fixed $k$, $G_{\leq k}(z)$ denote the generating function of lambda-terms where all
bindings have unary lengths not larger than $k$. Then
\begin{equation}\label{z:eq:3}
[z^n] G_{\leq k}(z) \sim \sqrt{\frac{2k + \sqrt{k}}{4 \pi \prod_{j=2}^{k+1} c_j}} n^{-\frac{3}{2}}
(1 + 2 \sqrt{k})^n, \quad \textrm{as } n \to \infty,
\end{equation}
where
\begin{equation} \label{crec}
c_{1} = 1 \quad \textrm{ and } \quad c_j = 4 j - 5 + 2 \sqrt{c_{j-1}}, \text{ for } 2 \leq
j \leq k+1.
\end{equation} 
\end{theorem}

\begin{proof}
Lemma~\ref{z:l:1} tells us that the dominant singularity $\hat\rho_k = \frac{1}{1+2\sqrt{k}}$ is
algebraic and of type $\frac12$. Hence, we get the factor $n^{-\frac{3}{2}} (1+2\sqrt{k})^n$ in
Eq.~(\ref{z:eq:3}).

Let us now consider the constant (w.r.t. $n$) term of Eq.~(\ref{z:eq:3}).
We have seen in Proposition~\ref{z:pr:1} that for $z$ close to $\hat\rho_k$, and with the notations used in its proof, $\hat{R}_{k+1,k}(\hat\rho_k (1-\epsilon)) = c_{k+1} \hat\rho_k^2 + d_{k+1} \sqrt{\epsilon} + \BigO{\epsilon}$. Since $G_{\leq k}(z) = \frac{1}{2z} \left(1 - \sqrt{\hat{R}_{k+1,k}(z)}\right)$, we get
\[
G_{\leq k}(\hat\rho_k (1-\epsilon)) = \frac{1-\sqrt{c_{k+1}}}{2} - \frac{d_{k+1}}{4 \hat\rho_k^2 \sqrt{c_{k+1}}} \sqrt{\epsilon} + \BigO{\epsilon}
\]
which gives 
\begin{align*}
[z^n] G_{\leq k}(z) \sim - \frac{d_{k+1}}{4 \hat\rho_k^2 \sqrt{c_{k+1}}} [z^n] \sqrt{1 -
\frac{z}{\hat\rho_k}}, \text{ as }\nti.
\end{align*}
Using the formulas $d_{k+1} = \frac{4 \hat\rho_k^{\frac32} k^{\frac14}}{\prod_{l=2}^{k}
\sqrt{c_l}}$ and $\hat\rho_k = \frac{1}{1+2\sqrt{k}}$ and then simplifying, we obtain the formula for the constant term.
\end{proof}

\subsection{Asymptotic decrease of constant term}

\begin{prop}\label{z:prop:decrease_const}
The multiplicative constant in \eqref{z:eq:3} satisfies  
\[
\sqrt{\frac{2k + \sqrt{k}}{4 \pi \prod_{j=2}^{k+1} c_j}} =
\frac{1}{D 2^{k+1} e^{\sqrt{k+1}}} \sqrt{\frac{(k+1)^{\frac14} \left(2k + \sqrt{k}\right)}{k!}}
\left(1 + \BigO{\frac{1}{\sqrt{k}}}\right), \text{ as } k\to\infty,
\]
where $D = \sqrt{\pi \omega e^{\frac14 - \frac54 \gamma + \zeta\left(\frac12\right)}}$ and
$\omega$ is a computable constant with numerical value $\omega\approx0.118\dots$.
\end{prop}

The proof of Proposition~\ref{z:prop:decrease_const} is focused on obtaining the asymptotic
expansion of the product $\prod_{j=2}^{k+1} c_j$ as $k \to \infty$.

\begin{lemma}
For $M\to\infty$ we have 
\[
\prod_{j=2}^{M} c_j = C M! 4^{M-1} e^{2 \sqrt{M}} M^{-\frac54} \left(1 + \BigO{\frac{1}{\sqrt{M}}}\right)
\]
where $C$ is a suitable constant.
\end{lemma}
\begin{proof}
From the recursive relation \eqref{crec} and by bootstrapping we obtain the asymptotic expansion
\[
c_j = 4j + 4\sqrt{j} - 3 -\frac{4}{\sqrt{j}} - \frac{1}{j} + \BigO{\frac{1}{j^{\frac32}}},
\text{ as } j\to\infty, 
\]
which we can rewrite as $c_j = (4j + 4\sqrt{j} - 3)(1 + \omega_j)$, where $\omega_j = \Theta\left(n^{-\frac32}\right)$.
Consider now the product $\prod_{j=2}^{M} c_j$ for $M$ large -- we shall take $M = k+1$ later on.
We write it as $\prod_{j=2}^{M} (4j + 4\sqrt{j} - 3) \cdot \prod_{j=2}^{M} (1 + \omega_j)$ and
consider each of the products separately.
\begin{itemize}
	\item $\prod_{j=2}^{M} (1 + \omega_j)$: This product has a finite limit
		$\omega$ if the series $\sum_j \omega_j$ is convergent, which is indeed the case.
		This limit can be computed numerically as 
		$\lim_{M\to\infty}\prod_{2 \leq j \leq M} \frac{c_j}{4j+4\sqrt{j}-3}$. However,
		the convergence is slow. The best we have got from the numerical studies is
		$\omega = 0.118\ldots$
	
	\item $\prod_{j=2}^{M} (4j + 4\sqrt{j} - 3)$: This product gives us the asymptotic
		behaviour. Let us rewrite it as 
\[
	M! 4^{M-1} \prod_{j=2}^{M} \left(1 + \frac{1}{\sqrt{j}} - \frac{3}{4 j}\right) = M!
	4^{M-1} \exp\left( \sum_{j=2}^{M} \log\left(1 + \frac{1}{\sqrt{i}} - \frac{3}{4 j}\right) \right).
\]
Now, knowing that $\log\left(1 + \frac{1}{\sqrt{j}} - \frac{3}{4 j}\right) = \frac{1}{\sqrt{j}} -
\frac{5}{4 j} + \BigO{\frac{1}{j^{\frac32}}}$, we can compute our sum as 
\begin{align*}
\sum_{j=2}^{M} \frac{1}{\sqrt{j}} - \frac{5}{4 j} + \BigO{\frac{1}{j^{\frac32}}} =& 2 \sqrt{M} - \frac{5}{4} H_M + \zeta\left(\frac12\right) + \frac14 + \BigO{\frac{1}{\sqrt{M}}} \\
=& 2 \sqrt{M} - \frac{5}{4} \log M + \left(\frac14 - \frac54 \gamma + \zeta\left(\frac12\right)\right) + \BigO{\frac{1}{\sqrt{M}}},
\end{align*}
where $H_M$ is the $M$th harmonic number and $\gamma = 0.57721\ldots$ is the Euler--Mascheroni constant. 
We finally obtain
\[
\prod_{j=2}^{M} (4j + 4\sqrt{j} - 3)= C M! 4^{M-1} M^{-\frac54} \exp^{2\sqrt{M}} \left(1 +
\BigO{\frac{1}{\sqrt{M}}}\right),
\]
where $C = \omega \exp^{\frac14 - \frac54 \gamma + \zeta\left(\frac12\right)}.$ 
\end{itemize}

Putting all pieces together we get the following formula for the constant term of Eq.~(\ref{z:eq:3})
\[
\sqrt{\frac{2k + \sqrt{k}}{4 \pi \prod_{j=2}^{k+1} c_j}} =
\frac{1}{D 2^{k+1} e^{\sqrt{k+1}}} \sqrt{\frac{(k+1)^{\frac14} \left(2k + \sqrt{k}\right)}{k!}} \left(1 + \BigO{\frac{1}{\sqrt{k}}}\right),
\]
where $D = \sqrt{\pi \omega e^{\frac14 - \frac54 \gamma + \zeta\left(\frac12\right)}}$.
\end{proof}

\section{Enumeration of lambda-terms of bounded unary height}
\label{sec:bounded-height}

We now turn to the enumeration of lambda-terms with bounded unary height. 

Let $\cH_{\leq k}$ denote the class of closed lambda-terms with unary height less than or equal to
$k$. Our first goal is to set up an equation for the $\cH_{\leq k}$.
Define the class $\mathcal{P}^{(i,k)}$ as the class of unary-binary trees such that $i+h_u(e)\le
k$ for every leaf $e$  ({\em i.e.} the unary height of every leaf $e$ is at most $k-i$) and every
leaf $e$ is colored with one out of $i+h_u(e)$ colors. 

As in the previous section, we observe that $\mathcal{P}^{(k,k)}$ is the class of all Motzkin tree
with $k$ types of leaves and $\mathcal{P}^{(0,k)}$ is isomorphic to the class
$\cH_{\leq k}$. The class $\mathcal{P}^{(1,k)}$ is isomorphic to the class obtained from
$\cH_{\leq k}$ by allowing free leaves. This class in turn is isomorphic to the class of closed
lambda-terms with a unary root: Just add a unary node as new root to a term of the previous class
and bind all free leaves by this newly added abstraction. 

For general $i$, $\mathcal{P}^{(i,k)}$ is isomorphic to the class of closed lambda-terms built as
follows: Consider a path of $i$ unary nodes to which we append a Motzkin tree with unary height
less than or equal to $k-i$ and call this structure the skeleton. Then, for each leaf $e$ there
are $i+h_u(e)$ way to bind it in order to make a closed lambda-term out of the skeleton. 

The classes $\mathcal{P}^{(i,k)}$ can be recursively specified, starting from a class $\cZ$ of
atoms, by 
$$
\mathcal{P}^{(k,k)}=k\mathcal{Z}+(\cA \times \mathcal{P}^{(k,k)} \times \mathcal{P}^{(k,k)})
$$
and, for $i<k$, by
\begin{equation}
\label{def-Pik}
\mathcal{P}^{(i,k)}=i\mathcal{Z}+ (\cA \times \cP^{(i,k)} \times \cP^{(i,k)}) + (\cU \times \mathcal{P}^{(i+1,k)}).
\end{equation}

Translating into generating functions we obtain 
$$ 
P^{(k,k)}(z)={\frac {1-\sqrt {1-4k{z}^{2}}}{2z}}
$$
and 
\begin{equation} \label{wherenestedrootscomefrom} 
P^{(i,k)}(z)={\frac {1-\sqrt {1-4i{z}^{2}-4{z}^{2}P^{(i+1,k)}(z)}}{2z}}, 
\end{equation} 
for $i<k$. 

Due to the remarks above, the recursive specification gives directly the generating function
$H_{\leq k}(z)= P^{(0,k)}(z)$ associated with $\cH_{\leq k}$. 
We get an expression involving $k+1$ nested radicals:
\begin{equation}
H_{\leq k}(z)={\frac {1-\sqrt {1-2z +2z\sqrt {\cdots 
\sqrt{ 1-4(k-i+1) z^2-2z+ 2z \sqrt{
\cdots
+2z \sqrt {1-4k{z}^{2}}  }   }}  }}{2z}}.
\end{equation}
Note that for $n\le k$ we have $[z^n]H_{\leq k}(z)=[z^n]L(z,1)$ and thus 
$H_{\leq k}(z)$ converges to $L(z,1)$ in the sense of formal convergence of power series
({\em cf.} \cite[p. 731]{FlSe09}).

In the next subsection we consider the singularities of this generating function and determine its dominant one together with its type -- we shall see that the location and the number of the dominant radicands changes with~$k$. 
Then we use this information to obtain the asymptotic behaviour of its coefficients. 
In Sections~\ref{sec:Motzkin} and~\ref{sec:window} we have seen examples where the dominant
radicand is either the innermost one, the outermost one, or all radicands together. 
We know of no previous example where the position of the dominant radicand changes depending on 
the number of levels of nesting.

\subsection{Analysis of the radicands}

We now consider how to determine the dominant singularity of the function $H_{\leq k}(z)$: It is
again built of nested radicals, hence its singularities are the values where at least one of the
radicands vanishes. 

Theorem~\ref{MainRadicand} below gives the dominant radicand in $H_{\leq k}(z)$, {\em i.e.}, the
radicand whose zero is the dominant singularity of $H_{\leq k}(z)$. But first, we introduce two
auxiliary sequences which prove to be important in the sequel. 

\begin{definition}\label{def_N}
Let $(u_i)_{i \geq 0}$ be the integer sequence defined by 
\[
	u_0=0 \text{ and } u_{i+1} = u_i^2 +i+1, \text{ for } i\ge 0
\]
and $(N_i)_{i \geq 0}$ by 
\[
	N_i=u^2_i-u_i+i, 
\]
for all $i\ge 0$.
\end{definition}

\begin{corollary}
\label{cor:rec-Ni}
The sequence $(N_i)_{i \geq 0}$ can be written without reference to the sequence $(u_i)_{i\ge 0}$ by
$N_0=0$, $N_1=1$ and $N_{i+1} = N_i^2 + 3 N_i +2 + (N_i+1) \sqrt{4 N_i - 4 i + 1}, \text{ for } i
\geq 1$. 

\end{corollary}
\begin{proof}
Solve the equation $N_i = u_i^2 - u_i +i$, considered as a quadratic equation in $u_i$, then plug its solution into the recursive definition for $u_{i+1}$.
This requires a little care, as the choice of the solution for expressing $u_i$ in terms of $N_i$ differs for $i=0$ and in the case $i \geq 1$.
\end{proof}

\begin{remark}
Obviously, the two sequences $(u_i)_{i \geq 0}$ and $(N_i)_{i \geq 0}$ are strictly increasing and
have super-exponential growth. Since the growth rate will be important for our analysis, we will
turn to it later. 
\end{remark}

\begin{theorem}
\label{MainRadicand}
Let $(N_i)_{i\ge 0}$ be the sequence defined in Def.~\ref{def_N} and $k$ be an integer. Define $j$
as the integer such that $k\in[N_j,N_{j+1})$. 
If $k\neq N_j$, then the dominant radicand of $H_{\leq k}(z)$ is the $j$-th radicand
(counted from the innermost one outwards), and the
dominant singularity $\rho_k$ is of type $\frac12$.
Otherwise, the $j$-th and the $(j+1)$-st radicand vanish simultaneously at the dominant singularity of $H_{\leq k}(z)$, which is equal to $1/(2 u_j)$ and of type $\frac14$.
\end{theorem}

The rest of this section is devoted to the proof of Theorem~\ref{MainRadicand}.

\subsubsection{The radicands $R_{i,k}$}

Let us denote by $R_{i,k}(z)$ the $i$th 
radicand ($1\leq i\leq k+1$) of $H_{\leq k}(z)$, according to the numbering from the innermost
outwards as adopted in the assertion of Theorem~\ref{MainRadicand}, {\em i.e.}, we have 
\begin{equation}\label{eq:p_and_r}
P^{(i,k)}(z)=\frac{1-\sqrt{R_{k-i+1,k}(z)}}{2z}. 
\end{equation}
We can write the radicands recursively as follows:
\[R_{1,k}(z):=1-4k{z}^{2}\]
and 
\begin{equation}
R_{i,k}(z):=1-4(k-i+1){z}^{2}-2z+2z\sqrt{R_{i-1,k}(z)},
\label{eq:unaires-bornes-Rik}
\end{equation}
for $2 \leq i \leq k+1$, which gives
\[
 R_{i,k}(z)= 1-4(k-i+1){z}^{2}-2z 
+2z \sqrt {1-\cdots \sqrt {1-4\left(k-1  \right) {z}^{2}-2z +2z \sqrt {1-4k{z}^{2}}  }   }.
\]

As $H_{\leq k}(z) = P^{(0,k)}(z) = (1-\sqrt{R_{k+1,k}(z)}\,)/(2z)$, the dominant singularity of
$H_{\leq k}(z)$ is the dominant singularity of $\sqrt{R_{k+1,k}(z)}$ as well.

\subsubsection{The dominant singularity of a radicand}

We show below that, for any fixed $k$ and for any $i$, $1 \leq i \leq k+1$, the $i$th radicand
$R_{i,k}$, when restricted to the real part of its definition domain, is decreasing and use this
to determine the interval where it is positive and to prove that it has a single real positive
root, which turns out to be the dominant singularity.

\begin{lemma}
\label{lemma:decreasingR}
For every $k>0$ and $1\leq i\leq k+1$, the real function $R_{i,k}(z)$ is strictly decreasing on
the positive real line (up to its first singularity). 
\end{lemma}

\begin{proof} 
The proof is a simple inductive argument like in Lemma~\ref{z:l:2}.
\end{proof}

\begin{corollary}
For every $k>0$ and $1\leq i\leq k+1$, the real function $R_{i,k}(z)$ has at most one real positive
root.
\end{corollary}

\begin{remark}\label{rootexistence}
If $j$ and $k$ are such that $k \in [N_j, N_{j+1})$, then it will turn out that only the $j+1$
first radicands $R_{1,k}(z),\ldots,R_{j+1,k}(z)$ will be relevant for our investigations.
All of them have a real positive root. This holds due to the fact
that $R_{j+1,k}(z)$ is a dominant radicand of $H_{\leq k}(z)$, which we shall prove
later on. 
\end{remark}

\begin{definition}
Let $j\ge 1$ and $k$ be integers such that $k \in [N_j,N_{j+1})$. 
For $i = 1,\ldots,j+1$ let $\sigma_{i,k}$ 
denote the smallest positive root of the radicand $R_{i,k}(z)$.
\end{definition}

\begin{lemma}\label{lemma:imaginary_roots}
Assume that the radical $\sqrt{R_{i,k}(z)}$ has (real) positive singularities 
and let $z_0>0$ be the smallest of them. Then there are no complex singularities with modulus
$|z_0|$. 
\end{lemma}

\begin{proof}
The proof is very similar to that of Lemma~\ref{nocomplexroots}. 
\end{proof}

The lemma guarantees that $H_{\leq k}(z)$ can have only one dominant singularity, which must be on
the positive real line.

Now we turn our attention to the list $(\sigma_{i,k})_{1 \leq i \leq j}$ where $j$ is such
that $N_j\le k<N_{j+1}$.

\begin{lemma}\label{lemma:decreasing-rho}
Let $j\ge 1$ and $k \in [N_j, N_{j+1})$ be given and assume that $\sigma_{i,k}$ and
$\sigma_{i+1,k}$ exist. Then we have $\sigma_{i+1,k} \leq \sigma_{i,k}$ for $1 \leq i \leq j$.
\end{lemma}
\begin{proof}
First note that, if $x_0$ is a singular point of some radical, then it is also a singular point of
all radicals which are lying more outwards. Therefore, if both function $R_{i,k}(z)$ and
$R_{i+1,k}(z)$ have positive roots, then, by definition, $\sigma_{i+1,k}$ is the smallest positive
root of $R_{i+1,k}(z)$. Hence, it is a singularity of $\sqrt{R_{i+1,k}(z)}$ and thus of
$R_{i,k}(z)$ as well. This immediately implies the assertion. 
\end{proof}

\begin{lemma}\label{Rkcomparison}
For any $i$ and $k$ the inequality $R_{i,k}(z) > R_{i,k+1}(z)$ holds for all $z>0$ for which the
two radicands are defined as real functions. 
\end{lemma}

\begin{proof}
Obviously, the assertion holds for $i=1$. Then, observe 
$$R_{i,k} - R_{i,k+1} = 4 z^2 + 2 z \left( \sqrt{R_{i-1,k}} - \sqrt{R_{i-1,k+1}} \right)$$ and
hence an easy induction completes the proof. 
\end{proof}

\subsubsection{When two successive radicands vanish}

\begin{lemma}
\label{lemma:cancellation-successive-radicands}
Assume that, for two indices $j$ and $k$ such that $1 \leq j \leq k$, the value $\sigma_{j,k}$,
which is a root of $R_{j,k}$, is also a root of the radicand $R_{j+1,k}$.
Then $\sigma_{j,k}=\frac{1}{1+\sqrt{1+4(k-j)}}$.
Moreover, $R_{j-p,k}(\sigma_{j,k}) = 4 \alpha_p \sigma_{j,k}^ 2$, for all $p < j$, where the sequence $\alpha_p$ is defined by 
\[
\left\{
\begin{array}{lcl}
\alpha_0 =& 0;
\\
\alpha_p =& (\alpha_{p-1} +p)^2 \;\;\; for \;\; p\geq 1.
\end{array}
\right.
\]
\end{lemma}
\begin{proof}
By our assumption, the two successive radicands $R_{j+1,k}$ and $R_{j,k}$ vanish for the same value $z=\sigma_{j,k}=\sigma_{j+1,k}$.
Therefore, from Eq.~(\ref{eq:unaires-bornes-Rik}) shifted from $j$ to $j+1$, 
we obtain that $1-4(k-j)z^2-2z = 0$,  
and this can only happen if  $\sigma_{j,k}$ is equal to 
$\frac{1}{1+\sqrt{1+4(k-j)}}$.

Now assume that $j \geq2$ and that $z=\sigma_{j,k}$, {\em i.e.} both $R_{j,k}(z)$ and $R_{j+1,k}(z)$ are equal to~0. Then
\begin{align*}
0 = R_{j,k}(z) = 1-4(k-j+1)z^2-2z + 2 z \sqrt{R_{j-1,k}(z)} 
= -4 z^2 + 2 z \sqrt{R_{j-1,k}(z)},
\end{align*}
and thus $R_{j-1,k}(z)= 4 z^ 2$.
Going one step further and assuming that $j \geq3$, we obtain that
\[
R_{j-1,k}(z) = 1-4(k-j+2)z^2-2z + 2 z \sqrt{R_{j-2,k}} = -8 z^2 + 2 z \sqrt{R_{j-2,k}}.
\]
Plugging the value $R_{j-1,k}(z) = 4 z^ 2$ into this equation gives $R_{j-2,k}(z) = 36 z^ 2$.
We iterate and obtain for $p \leq j-1$:
\begin{align*}
R_{j-p,k}(z)  = 1-4(k-j+p+1)z^2-2z + 2 z \sqrt{R_{j-p-1,k}}  
= -4 (p+1) z^2 + 2z \sqrt{R_{j-p-1,k}} 
\end{align*}
If $R_{j-p,k}(z) = 4 \alpha_{p} z^ 2$, then $R_{j-p-1,k}(z) = 4 \alpha_{p+1} z^2$ with $\alpha_{p+1} = (\alpha_p +p+1)^2$.
\end{proof}

\begin{remark} \label{rootexproved}
Note that Lemma~\ref{lemma:cancellation-successive-radicands} implies the existence of
$\sigma_{j,N_j}$ and $\sigma_{j+1,N_j}$. By Lemma~\ref{Rkcomparison} we have
$\sigma_{j+1,k}<\sigma_{j,k}$ and thus $\sigma_{j+1,N_j+\ell}$ exists for all $\ell\ge 0$. This
guarantees the existence of $\sigma_{i,k}$ for all $1\le i\le j+1$, as we claimed in
Remark~\ref{rootexistence}. 
\end{remark}

\begin{lemma}\label{lem_two_rads}
If the values $j$ and $k$ are such that there exists a value $z$ cancelling both radicands 
$R_{j+1,k}$ and $R_{j,k}$, then we must have $k = N_j$ where $(N_j)_{j\ge 0}$ is defined by $N_0 =
0$ and  $N_i := \alpha_i - \alpha_{i-1}$ for $i \geq 1$, with $(\alpha_i)_{i\ge 0}$ 
being the sequence 
defined in Lemma~\ref{lemma:cancellation-successive-radicands}.
\end{lemma}

\begin{remark}
The sequence $(N_j)_{j\ge 0}$ in Lemma~\ref{lem_two_rads} is precisely the sequence defined in
Def.~\ref{def_N}.
\end{remark}

\begin{proof}
From Lemma~\ref{lemma:cancellation-successive-radicands}, simultaneous vanishing of both radicands implies that $z=\sigma_{j,k}$.
Then we know the values of the $R_{j-p,k}(\sigma_{j,k})$ for all $p=0,\dots,j-1$; in particular, taking $p=j-1$ gives $R_{1,k} (\sigma_{j,k}) = 4 \sigma_{j,k}^2 \alpha_{j-1}$.
We have $R_{1,k} (z) = 1-4 k z^2$, which implies that
$1-4 k \sigma_{j,k}^2 = 4 \sigma_{j,k}^2 \alpha_{j-1}$. Hence we have  
$\sigma_{j,k}^2 = \frac{1}{4(k+\alpha_{j-1})}$.
But we also know that a suitable value $z=\sigma_{j,k}$ must be equal to
$\frac{-1+\sqrt{1+4(k-j)}}{4(k-j)}$, which gives an equation for the integers $k$ and $j$
involving also the sequence $(\alpha_j)_{j\ge 0}$ defined in Lemma~\ref{lemma:cancellation-successive-radicands}:
\begin{equation} \label{jkrelation}
\left( \frac{-1+\sqrt{1+4(k-j)}}{4(k-j)} \right)^2 =  
\frac{1}{4(k+\alpha_{j-1})}.
\end{equation} 
Setting $\ell =k-j$ and solving gives
$\ell = (j+\alpha_{j-1})(j+\alpha_{j-1}-1)$, which leads to 
$k=(j+\alpha_{j-1})^2 - \alpha_{j-1}$.
Finally, the recurrence for the $\alpha_i$ (see Lemma~\ref{lemma:cancellation-successive-radicands}) gives $k= \alpha_j - \alpha_{j-1}$.
\end{proof}

The first values of the $N_j$ are given by the table of Figure~\ref{fig:N(i)}.
For each value $k=N_j$, the two radicands that vanish are those numbered by $j$ and $j+1$.
\begin{table}
\centering
\begin{tabular}{c||c|c|c|c|c|c}
$j$ & 1 & 2 & 3 & 4&5&6
\\ \hline
$N_j$ & 1 & 8 & 135 & 21$\,$760 & 479$\,$982$\,$377 &23$\,$040$\,$411$\,$505$\,$837$\,$408
\\ \hline 
$u_j$ & 1 & 3 & 12 & 148 & 21$\,$909 & 480$\,$004$\,$287 
\end{tabular}
\caption{\label{fig:N(i)}  The first values $N_j$ and $u_j$.}
\end{table}
\begin{lemma}\label{nomorethantwo}
No more than two radicands can vanish at the same positive value. If so, then these two 
radicands are consecutive ones.
\end{lemma}
\begin{proof}
Assume that two non-consecutive radicands $R_{i,k}$ and $R_{j,k}$ vanish simultaneously. 
From Lemma~\ref{lemma:decreasing-rho}, we know that the zeroes of the radicands decrease.
Therefore, all the radicands $R_{\ell,k}$ for $i \leq \ell \leq j$ would vanish simultaneously.
But it is not possible that more than two successive nested radicands $R_{i,k},\dots, R_{i+p,k}$
have a common positive zero: This can only happen for $z = \sigma_{i,k}$, but then the polynomial
part $1-4 (k-j+1) z^2 - 2z$ can be simplified into $4 (j-i-1) z^2$, hence it is strictly positive
as soon as $j>i+1$. 
\end{proof}

\subsubsection{The sequence $(N_i)_{i \geq 1}$}

We establish here results about the growth of the sequence $(N_i)_{i \geq 1}$.

\begin{lemma}
\label{lemma:def-sequence-u}
The sequence $(u_i)_{i \geq 0}$ defined in Def.~\ref{def_N} satisfies 
$u_i = \alpha_i+i$. Moreover, the limit 
\[
\chi:=\lim_{i\to \infty} u_i^{1/2^i}  \approx 1.36660956\dots
\]
exists. Furthermore, we have $u_i=\lfloor \chi^{2^i} \rfloor$ for sufficiently large $i$. As a
consequence, both sequences $(u_i)_{i \geq 0}$ and $(N_i)_{i \geq 1}$ have 
doubly exponential growth.
\end{lemma}

\begin{proof}
The recurrence relation on the $u_i$ is clear from the definition of the $\alpha_i$ in Lemma~\ref{lemma:cancellation-successive-radicands}.

Aho and Sloane \cite{AS70} study doubly exponential integer sequences 
$\mathbf x=(x_i)_{i\geq 0}$ of the
form $x_{i+1} = x_i^2 + g_i $ with $|g_i| < x_i /4$ for $i$ sufficiently large.
They show there that for any such sequence $\mathbf x=(x_i)$ the limit $\chi_{\mathbf x}:=\lim_{i\to\infty} x_i^{1/2^i}$ exists and that the sequence can be written, for $i$ large enough,  as $x_i=\lfloor\chi_{\mathbf x}^{2^i}\rfloor$. 

In our case it is easy to check that, for $i \geq 4$, $g_i=i+1 < u_i / 4$. Hence the sequence
$(u_i)_{i\ge 0}$ is of a form such that the result of \cite{AS70} applies, and $\lim_{i\to \infty}
u_i^{1/2^i}$ can be numerically approximated by $\chi \approx 1.36660956\dots$ 

Finally, the relation $N_i = u_i^2 - u_i + i  = u_{i+1}-u_i-1$ 
implies that $(N_i)_{i \geq 0}$ is doubly exponential as well. Of course, since $\alpha_i =
u_{j+1} -j-1$, the sequence $(\alpha_i)_{i \geq 0}$ is also doubly exponential. 
\end{proof}

\begin{remark}
Note, however, that for neither of the sequences $(N_i)_{i \geq 1}$ and $(\alpha_i)_{i \geq 0}$
the result of Aho and Sloane \cite{AS70} can be applied (see the recurrences they satisfy, given 
by Corollary~\ref{cor:rec-Ni} and Lemma~\ref{lemma:cancellation-successive-radicands}).
\end{remark}

\subsubsection{The singularities}

The following proposition sums up the properties of the singularities.

\begin{prop}
\label{prop:summary-singularities}
\begin{enumerate}
\item[(i)]
Let $\rho_k$ be the dominant singularity of $H_{\leq k}(z)$ for $k=0,1,2, \dots$ 
Then the sequence $(\rho_k)_{k \geq 0}$ is strictly decreasing.
\item[(ii)]
If there exists a $j \geq 1$ such that $k = N_j$, then the dominant singularity 
$\rho_{N_j} = \sigma_{j,N_j}= \frac{1}{(2 u_j)}$ is a root of both radicands $R_{j,k}$ and
$R_{j+1,k}$, and it is of type $\frac{1}{4}$
\item[(iii)]
For $k \in \left( N_j, N_{j+1} \right)$, the dominant singularity is $\rho_k$ is a root of the
single radicand $R_{j,k}$; it is of type $\frac{1}{2}$ and lies in the interval $\left( \frac{1}{2 u_{j+1}}, \frac{1}{2 u_j} \right)$.
\end{enumerate}
\end{prop}

\begin{proof}
\begin{itemize}

\item[(i)]
If the $j$th radicand of $H_{\leq k}$ is dominant, then $R_{j,k}(\rho_k)=0$.
This implies that $R_{j,k+1}(\rho_k) < R_{j,k}(\rho_k) = 0$ and therefore $\rho_{k+1} <
\rho_k$, since the radicands are strictly decreasing functions by Lemma~\ref{lemma:decreasingR}. 

\item[(ii)]
If there exists a $j$ such that $k=N_j$, then the pair $(j,k)=(j,N_j)$ is a solution of
\eqref{jkrelation}. If we set $\sigma_{j,N_j} = \frac{-1+\sqrt{1+4(N_j-j)}}{4( N_j -j)}$, use
\eqref{jkrelation}, and then go backwards the steps in the proof of Lemma~\ref{lem_two_rads}, we
eventually arrive at $R_{j,k}(\sigma_{j,N_j})=R_{j+1,k}(\sigma_{j,N_j})=0$. The type of the 
singularity is an immediate consequence of the fact that the two dominant radicands are
consecutive ones. 

In order to obtain the last claim, note that 
$N_j-j = \alpha_j - \alpha_{j-1} -j = (\alpha_{j-1}+j)^2 - \alpha_{j-1}-j$ and $1+4(N_j-j) = (1-2
(\alpha_{j-1}+j))^2 =(1-2 u_j)^2$, which gives, after simplification and choosing the root that is
positive and has smallest modulus, $\sigma_{j,N_j} = 1/(2 u_j)$.

\item[(iii)]
For $N_j < k < N_{j+1}$, Lemmas~\ref{lem_two_rads} and~\ref{nomorethantwo} tell us that no two
radicands vanish simultaneously; only the $j$th radicand is the dominant one
and the singularity is therefore of type $1/2$. 
The bounds for $\sigma_{j,k}$ follow from the fact that for any given value of $j$ the sequence of zeroes of $R_{j,k}(z)$ is decreasing 
(see Lemma~\ref{Rkcomparison} and Remark~\ref{rootexproved}). \qedhere
\end{itemize}
\end{proof}

\bigskip
The sequence of the dominant singularities for $k \in \{ N_j \,\mid \,j\ge 1\}$ is 1/2, 1/6, 1/24, 1/296, 1/43818, 1/960008574, 1/460808231076756752, \dots

As a corollary, we get the well-known result that $L(z,0)$ only converges at $z=0$, which follows
from \cite{BoGaGi11} or the estimates given in \cite[Section~5]{BGGJ13}.

\begin{corollary}
\label{RCglobal}
The radius of convergence of the generating function $L(z,0)$ enumerating all lambda-terms is zero.
\end{corollary}
\begin{proof}
The number of lambda-terms of given size $n$ being greater than the number of lambda-terms of size
$n$ and unary height $p$ (for any~$p$), the radius of convergence of the global generating
function $L(z,0)$ must be smaller than (or equal to) the radius of convergence $\rho_k$ of the
function $H_{\leq N_k}(z)$, for any~$k$. But the sequence of these radii is the sequence
$(\frac{1}{2 u_k})$ and  converges to~0.
\end{proof}

\subsection{Asymptotic analysis and transition between different behaviours}

\subsubsection{Behaviour of the radicands}

In order to proceed, we need some information on the behaviour of the radicands in a neighbourhood
of the dominant singularity. This is done in the two propositions that follow:
Proposition~\ref{prop:radicands-at-rho} gives the exact values and
Proposition~\ref{prop:values-radicands} their expansions at the singularity.

\begin{prop}
\label{prop:radicands-at-rho}
The values of $R_{s, N_j}(z)$ at $z= \sigma_{j,N_j}$ are as follows: 
\begin{itemize}
\item[(i)] If $s < j$ (inner radicands), then, with the $u_j$ as defined in Lemma~\ref{lemma:def-sequence-u}
\[
R_{s,N_j} \left( \sigma_{j,N_j} \right) =
\left( \frac{u_{j-s}}{u_j} \right)^2. 
\]

\item[(ii)] If $s=j$ or $s=j+1$, then $R_{j,N_j} (\sigma_{j,N_j}) = R_{j+1,N_j}(\sigma_{j,N_j}) =
0$.

\item[(iii)]
If $j+1 < s $ (outer radicands), then 
\[
R_{s,N_j} \left( \sigma_{j,N_j} \right) = \frac{\lambda_{s-j-1}}{u_j^2},
\]
with the sequence $\lambda_\ell$ defined by $\lambda_0 = 0$ and $\lambda_{\ell+1} = \ell +1
+\sqrt{\lambda_\ell}$ for $\ell\ge 0$.
\end{itemize}
\end{prop} 

\begin{proof}
\begin{itemize}
\item[(i)]
The first assertion comes from Lemma \ref{lemma:cancellation-successive-radicands}, which gives 
$R_{s,j}(\sigma_{j,N_j}) = 4 \alpha_{j-s} \sigma_{j,N_j}^2$, 
and from Lemma~\ref{lemma:def-sequence-u}, from which we have 
$\alpha_{j-s} = u_{j-s}^2$.

\item[(ii)]
The second assertion is simply the definition of~$\sigma_{j,N_j}$.

\item[(iii)]
For the case $s>j+1$, we first check, using the equality $1-4 (N_j-j) \sigma_{j,N_j}^2 - 2 \sigma_{j,N_j} = 0$, that 
\[
R_{j+2,N_j} (\sigma_{j,N_j}) = 
1 - 4 (N_j-j-1) \sigma_{j,N_j}^2 - 2 \sigma_{j,N_j} + 2 \sigma_{j,N_j} \sqrt{R_{j+1,N_j}(\sigma_{j,N_j})}
= 4 \sigma_{j,N_j}^2.
\]
Now assume that for some $\ell \geq2$ we have   
$R_{j+\ell, N_j}(\sigma_{j,N_j}) = 4 \lambda_{\ell-1} \sigma_{j,N_j}^2$ and proceed by induction
(we have just checked that it holds for $\ell =2$ with $\lambda_1 = 1$).
Then
\begin{align*}
R_{j+\ell+1,N_j} (\sigma_{j,N_j}) =&
1 - 4 (N_j -j-\ell) \sigma_{j,N_j}^2 - 2 \sigma_{j,N_j} + 2 \sigma_{j,N_j} \sqrt{R_{j+\ell,N_j}(\sigma_{j,N_j})}
\\ =&
4 \sigma_{j,N_j}^2 (\ell +  \sqrt{\lambda_{\ell -1}})
= 4 \lambda_\ell \sigma_{j,N_j}^2
\end{align*}
again from the fact that $1-4 (N_j-j) \sigma_{j,N_j}^2 - 2 \sigma_{j,N_j} = 0$, and from the
recurrence assumption on $R_{j+\ell, N_j}(\sigma_{j,N_j})$. \qedhere
\end{itemize}
\end{proof}

\begin{prop}
\label{prop:values-radicands}
Let $\rho=\sigma_{j,N_j}$ be the dominant singularity of $H_{\leq N_j}(z)$.
Then, for any $\epsilon >0$
\begin{itemize}
\item[(i)]
	\begin{equation} \label{gammaeq}
R_{j,N_j} (\rho-\epsilon) = \gamma_j \epsilon + O \left( \epsilon^2 \right)
\qquad
with
\qquad
\gamma_j= - \frac{d}{dz} R_{j,N_j} (\rho).
\end{equation} 

\item[(ii)]
\begin{equation} \label{deltaeq}
	R_{j+1,N_j} (\rho-\epsilon) = 2\rho\sqrt{\gamma_j} \,\epsilon^\frac{1}{2} + O (\epsilon), 
\end{equation} 

\item[(iii)]
for $p \geq2$, 
\[
R_{j+p,N_j} (\rho-\epsilon) = 4 \rho^2 \lambda_{p-1} + \frac{(2 \rho)^\frac{3}{2} \gamma_j^\frac{1}{4}}{2^{p-2}  \sqrt{ \prod_{i=1}^{p-2} \lambda_i}}\; \epsilon^\frac{1}{4}  +  O \left( \epsilon^\frac{1}{2} \right) 
\]
where the sequence $(\lambda_i)_{i \geq 1}$ is defined in Proposition~\ref{prop:radicands-at-rho}.
\end{itemize}
\end{prop}

\begin{proof}
We know that $R_{j-1,N_j}(\sigma_{j,N_j})>0$ and that the function $R_{j-1,N_j}(z)$ is
analytic up to some value $z>\rho$. Hence $R_{j,N_j}(z)$ itself has a Taylor expansion around
$\rho$ which yields \eqref{gammaeq}. Using the recurrence relation \eqref{R_k_nesting} for
$R_{j,k}(z)$ we immediately obtain \eqref{deltaeq}.

The next step is computing the expansion of $R_{j+2,N_j}$ around $\rho$ where it has a singularity
of type $\frac{1}{4}$. We obtain 
\[
	R_{j+2,N_j} (\rho - \epsilon) = 4 \rho^2 + 2 \rho \sqrt{2\rho\sqrt{\gamma_j}} \, \epsilon^\frac{1}{4} + O \left( \epsilon^\frac{3}{4} \right).
\]

Now consider the  radicands $R_{j+p,N_j}(z)$ for $p \geq 2$ and proceed by induction:
They have a common dominant singularity at $z=\rho$, which is of type $\frac{1}{4}$.  
Thus, for all $p \geq 2$, there exist $a_p \neq 0$ and $b_p$ such that
$R_{j+p, N_j} (\rho - \epsilon) = a_p + b_p \epsilon^\frac{1}{4} +  O \left( \epsilon^\frac{1}{2}
\right).$ We already know that $a_2=4 \rho^ 2$ and $b_2 = 2 \rho \sqrt{2\rho\sqrt{\gamma_j}}$. By
the recurrence relation \eqref{R_k_nesting} for the radicands we get 
\[
R_{j+p+1,N_j} (\rho-\epsilon) = 1-4(N_j-j-p) (\rho-\epsilon)^2 - 2(\rho-\epsilon) +
2(\rho-\epsilon) \sqrt{R_{j+p,N_j}(\rho-\epsilon)}.
\]
Plugging in the expansion $a_p + b_p \epsilon^\frac{1}{4} +  O \left( \epsilon^\frac{1}{2} \right)$ for $R_{j+p, N_j} (\rho-\epsilon)$, expanding and simplifying the constant term through $1-4(N_j -j) \rho^2 - 2 \rho = 0$ gives
\[
R_{j+p+1, N_j} (\rho-\epsilon) = 4 p  \rho^ 2 + 2 \rho  \sqrt{a_p} + \frac{\rho  b_p}{\sqrt{a_p}} \epsilon^\frac{1}{4} +  O \left( \epsilon^\frac{1}{2} \right).
\]
Setting $a_{p+1} = 4 p  \rho^ 2 + 2 \rho  \sqrt{a_p}$ and $b_{p+1} = \frac{\rho  b_p}{\sqrt{a_p}}$,
we obtain $R_{j+p+1, N_j} (\rho-\epsilon) = a_{p+1} + b_{p+1} \epsilon^\frac{1}{4} +  O \left(
\epsilon^\frac{1}{2} \right).$ 
\begin{enumerate}
\item[--]
By dividing the recurrence for $a_p$ by $4 \rho^ 2$, we see that 
$\frac{a_{p+1}}{4 \rho^2} = p + \sqrt{\frac{a_p}{4 \rho^2}}$.
Coupled with $a_2 = 4 \rho^2$ and the definition of the $\lambda_i$, this gives
$a_p = 4 \rho^ 2  \lambda_{p-1}$.
\item[--]
Plugging the expression for $a_p$ that we have just obtained into the recurrence for the $b_p$ gives $b_{p+1} = \frac{b_p}{2 \sqrt{\lambda_{p+1}}}$ and finally
\[
b_p = \frac{b_2}{2^{p-2}  \sqrt{ \prod_{i=1}^{p-2} \lambda_i}}
\qquad \text{with} \qquad
b_2 = (2 \rho)^\frac{3}{2}  \gamma_j^\frac{1}{4}.
\qedhere
\]
\end{enumerate}
\end{proof}

\subsubsection{Asymptotic number of lambda-terms of bounded height}

We are now in the position to give the asymptotic behaviour of the number of lambda-terms with bounded unary height.
\begin{theorem}
\label{theo-4.1}
Let $(N_i)_{i\ge 0}$ and $(u_i)_{i\ge 0}$ be as in Def.~\ref{def_N}.
\begin{itemize}
\item[(i)]
If there exists $j \geq 0$ such that $N_{j}<k< N_{j+1}$, then there exists a constant $h_k$ such that
\begin{equation}  \label{firstasym}
[z^n] H_{\leq k}(z)\sim h_k n^{-3/2}(\sigma_{j,k})^{-n}, \mbox{ as }n\to\infty.
\end{equation} 

\item[(ii)]
If there exists $j$ such that $k=N_j$, then the following asymptotic relation holds: 
\begin{equation} \label{secondasym}
[z^n ] H_{\leq N_j}(z)\sim h_k n^{-5/4}(\sigma_{j,k})^{-n} = h_{N_j} n^{-5/4}(2u_j)^n, 
\mbox{ as }n\to\infty, 
\end{equation} 
where  
\begin{equation} \label{hNj}
h_{N_j} = \frac { \gamma_j^{1/4} (2u_j)^{1/4}}
{2^{N_j-j+2} \sqrt{2} \, \Gamma(3/4) \sqrt{\prod_{i=1}^{N_j-j} \lambda_i}},
\end{equation} 
with $\gamma_j$ and the sequence $(\lambda_i)_{i\ge0}$ as defined in 
Proposition~\ref{prop:radicands-at-rho}.
\end{itemize}
\end{theorem}

\begin{proof}
The expressions given in \eqref{firstasym} and \eqref{secondasym} follow immediately from
the fact that the dominant singularity for the cases $k\neq N_j$ and $k=N_j$ is of type $1/2$ and
$1/4$, respectively, and then applying the transfer theorem of Flajolet and Odlyzko \cite{FO}. 
What is left to do is proving \eqref{hNj}. 

(ii) If $k=N_j$, then from Proposition~\ref{prop:values-radicands} and 
$H_{\leq N_j}(z) = \frac{1}{2z} \left( 1-\sqrt{R_{N_j+1,N_j}}(z)\right)$ we get
\[
H_{\leq N_j}(\rho-\epsilon) = 
\frac{1-\sqrt{a_{N_j-j+1}}}{2\rho} - \frac{b_{N_j-j+1}}{4 \rho \sqrt{a_{N_j-j+1}}} \, \epsilon^\frac{1}{4} + O \left( \epsilon^\frac{1}{2} \right)  
\]
which gives (using again \cite{FO})  
\begin{align*}
[z^n] H_{\leq N_j} (z) \sim &
	- \frac{b_{N_j-j+1}}{4 \rho \sqrt{a_{N_j-j+1}}} (2u_j)^n  [z^n] (\rho-z)^\frac1{4}
\\ \sim &
	- \frac{b_{N_j-j+1}}{4 \rho^{\frac{3}{4}}  \sqrt{a_{N_j-j+1}}} (2u_j)^n \frac{n^{-\frac{5}{4}}}{\Gamma(- \frac{1}{4})}.
\end{align*}
Finally, plug in the expressions of $a_{N_j-j+1}$ and $b_{N_j-j+1}$, then simplify 
using also $\Gamma(-\frac{1}{4}) = -4 \Gamma(\frac{3}{4})$, to obtain the expression of~$h_{N_j}$.
\end{proof}

\subsection{The location of singularities for large $k$}

In this section we would like to investigate the sequence $(\rho_k)_{k\ge 0}$ itself. 

Let us first derive a few auxiliary results that we will need in order to proceed with the
analysis of the asymptotic behaviour of $\rho_{k}$ as $k\to\infty$.

\begin{prop}\label{z:sing:pr1}
If $\rho_{k}$ denotes the dominant singularity of $H_{\leq k}(z)$, then $\rho_{k} \geq \frac{1}{1+2\sqrt{k}}$.
\end{prop}
\begin{proof}
Let us recall that $\mathcal{G}_{\leq k}$ is the class of closed lambda-terms where all bindings
have unary length less than or equal to $k$, $G_{\leq k}(z)$ its generating function and $\hat{\rho}_{k} = \frac{1}{1+2\sqrt{k}}$ the dominant singularity of $G_{\leq k}(z)$.

Clearly, $\mathcal{H}_{\leq k} \subseteq \mathcal{G}_{\leq k}$ and therefore exponential growth of
$G_{\leq k}(z)$ is not larger than the exponential growth of $H_{\leq k}(z)$, {\em i.e.} $\rho_{k} \geq \hat{\rho}_{k}.$ 
\end{proof}

\begin{prop}\label{z:sing:pr2}
For $i = \BigO{\log\log k}$ we have
\[
R_{i,k}\left(\frac{1}{1+2\sqrt{k}}\right) = \frac{k^{2^{-i}}}{k} + \BigO{\frac{\log\log k}{k}},
\quad \textrm{as } k \to \infty.
\]
\end{prop}
\begin{proof}
We prove the assertion by induction on $i$:
\[
R_{1,k}\left(\frac{1}{1+2\sqrt{k}}\right) = 1 - \frac{4 k}{(1+2\sqrt{k})^2} = k^{-\frac12} + \BigO{\frac{1}{k}}.
\]
Now, assume that for some $i = \BigO{\log\log k}$, $R_{i,k}\left(\frac{1}{1+2\sqrt{k}}\right) = \frac{k^{2^{-i}}}{k} + \BigO{\frac{\log\log k}{k}}$ then
\[
R_{i+1,k}\left(\frac{1}{1+2\sqrt{k}}\right) = 1 - \frac{4(k-i)}{(1+2\sqrt{k})^2} - \frac{2}{1+2\sqrt{k}} + \frac{2}{1+2\sqrt{k}} \sqrt{\frac{k^{2^{-i}}}{k} + \BigO{\frac{\log\log k}{k}}}
\]
but it is easy to see that $i = \BigO{\log\log k}$
\begin{align*}
\frac{2}{1+2\sqrt{k}} =& \frac{1}{\sqrt{k}} + \BigO{\frac{1}{k}},\\
1 - \frac{4(k-i)}{(1+2\sqrt{k})^2} - \frac{2}{1+2\sqrt{k}} =& \BigO{\frac{\log\log k}{k}}.
\end{align*}
Thus, we can finish the proof with the following calculations:
\begin{align*}
R_{i+1,k}\left(\frac{1}{1+2\sqrt{k}}\right) =& \sqrt{k^{2^{-i}-2} + \BigO{\frac{\log\log k}{k^2}}} + \BigO{\frac{\log\log k}{k}} \\
=& k^{2^{-(i+1)}-1} + \BigO{\frac{\log\log k}{k}}.  \qedhere
\end{align*}
\end{proof}

\begin{prop}\label{z:sing:pr3}
If $j$ is such that $R_{j,k}(z)$ is a dominant radicand of the generating function
$H_{\leq k}(z)$, then $j = \BigO{\log\log k}$.
\end{prop}
\begin{proof}
Let us first consider the case where both the $j$th and the $(j+1)$st radicand are
dominant. From Theorem~\ref{MainRadicand} we know that in that case $k = N_j = u_{j+1}-u_j-1$.
Moreover, from Lemma~\ref{lemma:def-sequence-u} we have $u_i=\lfloor C^{2^i}\rfloor$ for
sufficiently large $i$ and with $C \approx 1.36660956\ldots$
Thus, $k = C^{2^{j+1}} \left(1 + \smallO{1}\right)$ and applying the logarithm twice on both sides
of this equation we get $j = \BigO{\log\log k}$.

In the case where $R_{j,k}(z)$ is the only dominant radicand we have $N_{j-1} < k < N_j$. 
It is enough to consider the left inequality $N_{j-1} = C^{2^{j}} (1 + \smallO{1}) < k$.
Proceeding like in the previous case we get $j = \BigO{\log\log k}.$
\end{proof}

We are now in the position to give the asymptotic behaviour of $\rho_{k}$.
\begin{theorem}\label{z:sing:l1}
Let $\rho_{k}$ be the dominant singularity of $H_{\leq k}(z)$, then the asymptotic behaviour of
$\rho_k$ can be described as follows:  
\begin{itemize}
\item If $k = N_j$ ($R_{j,k}(z)$ and $R_{j+1,k}(z)$ are dominant), then 
\begin{equation}\label{z:rho1}
\rho_{k} = \frac{1}{2\sqrt{k}} - \frac{1}{4k} + \BigO{\frac{\log\log k}{k^{\frac32}}}, \quad
\text{as } k \to \infty. 
\end{equation}
		
\item If $N_{j-1} < k < N_j$ (only $R_{j,k}(z)$ is dominant), then 
\begin{equation}\label{z:rho2}
\rho_{k} = \frac{1}{2\sqrt{k}} - \frac{1}{4k} + \BigO{\frac{1}{k^{\frac32 - \frac{1}{2^j}}}},
\quad \text{as } k \to \infty.
\end{equation}
\end{itemize}
\end{theorem}

\begin{proof}
Let us first consider the case where $k = N_j$. 
From Lemma~\ref{lemma:cancellation-successive-radicands}
and Proposition~\ref{prop:summary-singularities} we know that $\rho_{k} = \sigma_{j,k} =
\frac{1}{1 + \sqrt{1+4(k-j)}}.$ Proposition~\ref{z:sing:pr3} tells us that $j=O(\log\log k)$ and
thus expanding yields $\rho_{k} = \frac{1}{2\sqrt{k}} - \frac{1}{4k} +
\BigO{\frac{\log\log k}{k^{\frac32}}}$ as desired.

Proving the result for the case where $N_{j-1} < k < N_j$ is less straightforward.
Let us recall the result of Proposition~\ref{z:sing:pr1}: $\rho_{k} \geq \hat{\rho}_{k} =
\frac{1}{1+2\sqrt{k}} = \frac{1}{2\sqrt{k}} - \frac{1}{4k} + \BigO{\frac{1}{k^{\frac32}}}.$ So,
what is left is proving an upper bound. 

We have $\rho_{k} < \frac{1}{2\sqrt{k}}$, which is the value that cancels the 
innermost radicand $R_{1,k}(z) = 1 - 4 k z^2.$ 
Unfortunately, this upper bound is too weak to be used in this proof.

In order to improve the upper bound for $\rho_{k}$ notice that $\rho_{k}$ is a root of 
$R_{j,k}(z) = 1 - 4(k-j+1)z^2 - 2z + 2z \sqrt{R_{j-1,k}(z)}$ and that
$\sigma_{j-1,k}>\rho_k=\sigma_{j,k}$. This inequality can be seen as follows: The weak inequality
follows from Lemma~\ref{lemma:decreasing-rho}. But it is even strict, because no two successive
radicands can be zero. Thus the zeros $\sigma_{j-1,k}$ and $\sigma_{j,k}$ of the two respective
radicands must be different. 

Furthermore, we know that  
$R_{j-1,k}(z)$ is decreasing on the positive real axis (see
Lemma~\ref{lemma:decreasingR}) and that $\rho_{k} \geq \frac{1}{1+2\sqrt{k}}$. Thus, for $z\in
[\rho_k,\sigma_{j-1,k}]$ we have 
$R_{j-1,k}\left(\frac{1}{1+2\sqrt{k}}\right) \geq R_{j-1,k}(z)$ and $\overline R_{j,k}(z) \ge
R_{j,k}(z)$ where $\overline{R}_{j,k}(z) = 1 - 4(k-j+1)z^2 - 2z + 2z
\sqrt{R_{j-1,k}\left(\frac{1}{1+2\sqrt{k}}\right)}$. One can easily check that
$\overline{R}_{j,k}(z)$ is decreasing for $z>0$ and thus its positive root 
\[
\overline{\rho}_{k} = \frac{1}{1 - \sqrt{\kappa} + \sqrt{5 + 4(k-j) + \kappa - 2 \sqrt{\kappa}}}, 
\]
where $\kappa = R_{j-1,k}\left(\frac{1}{1+2\sqrt{k}}\right)$, must satisfy $\overline {\rho}_k\ge
\rho_k$. This inequality together with 
\[
\overline{\rho}_{k} = \frac{1}{2\sqrt{k}} - \frac{1}{4k} + \BigO{\frac{1}{k^{\frac32 -
\frac{1}{2^{j}}}}},
\]
where we used Proposition~\ref{z:sing:pr2} for the asymptotic expansion of $\kappa$ as $k \to
\infty$, completes the proof. 
\end{proof}

\subsection{Exponential decrease of the constant}

Numerical computations for the coefficients of asymptotic expansions when $k=1,8,135$ give
\begin{align*}
h_{1} =&  0.24261\ldots,
\\
h_{8} =&  9.31888\ldots \cdot 10^{-5},
\\
h_{135} =&  8.56995\ldots \cdot 10^{-157}.
\end{align*}
In Theorem~\ref{theo-4.1} 
we presented an expression for these constants (see Eq.~\eqref{hNj}) involving the
quantities $\gamma_j$ and $(\lambda_i)_{i\ge 0}$ which were defined in Proposition~\ref{prop:values-radicands}.
We now prove that the constant $h_{N_j}$ decreases exponentially fast as $j \to \infty$. 

\begin{prop}
\label{prop:decrease-h_Nj}
The constant $h_{N_j}$ satisfies, as $j\to \infty$,   
\begin{equation} \label{hNjsecondexpression}
h_{N_j} = D \, \frac {e^{\frac{1}{2} u_j^2 - u_j}} {(2 u_j)^{u_j^ 2-u_j} }
\left(1+O\left(\frac{1}{u_j}\right)\right), 
\end{equation} 
where 
\begin{equation} \label{constantD}
D = \frac{C^{1/4}}{\sqrt{\omega} \, e^{\frac{1}{2} \zeta(1/2) - \frac{1}{4}} 2^{5/2}
\Gamma(3/4) \pi^{1/4}} \approx 1.0506\dots
\end{equation} 
\end{prop}
The proof of Proposition~\ref{prop:decrease-h_Nj} starts from the value given in Eq.~\eqref{hNj}
and has two main parts: proving that $\gamma_j$ is of order $u_j$ and dealing with the product $\prod_{i=1}^{N_j-j} \lambda_i$.

\subsubsection{The derivative of $R_{j,k}(z)$}

Maple computations show that 
$ \frac{\gamma_j}{u_j}$ seems to converge quickly (with a precision of $10^{-10}$ for $j=7$) to a constant value, approximately equal to 6.347269145.
We will show that this indeed holds.

\begin{lemma}
\label{lemma:recurrence-gamma_j}
Define $w_{\ell,N_j} = \frac{d}{dz} R_{\ell,N_j} (\rho)$ with $\rho=\sigma_{j,N_j}$ as in the
previous section. For $p \geq 1$ set 
\[
\delta_{p,j} = - 4\frac{N_j-p+1}{u_j} -2 + 2 \frac{u_{j-p+1}}{u_j}
\quad\text{ and } \quad
\epsilon_{p,j} = \frac{1}{2 u_{j-p+1}}.
\]
Then
$w_{1,N_j} = -4 \frac{N_j}{u_j}$ and, for $p > 1$, 
\[
w_{p,N_j} = \sum_{s=1}^{p} \delta_{s,j} \; \prod_{r=s+1}^p \epsilon_{r,j}.
\]
\end{lemma}
\begin{proof}
The computation of $w_{1,N_j}$ is straightforward from $R_{1,N_j} (z) = 1-4 N_j z^2$ and $\rho = \frac{1}{2 u_j}$; note that $\delta_{1,j} = - 4\frac{N_j}{u_j}= w_{1,N_j}$.
Now for $p \geq2$ we have 
\[
R_{p,N_j}(z) = 1 - 4 (N_j-p+1) z^2 - 2 z + 2 z \sqrt{R_{p-1,N_j}(z)},
\]
which gives by derivation w.r.t $z$
\[
R_{p,N_j}^{'}(z) = -8 (N_j-p+1) z - 2  + 2 \sqrt{R_{p-1,N_j}(z)} + z \frac{R_{p-1,N_j}^{'}(z)}{\sqrt{R_{p-1,N_j}(z)}}. 
\]
Taking $z=\rho=\frac{1}{2 u_j}$, we get
\[
w_{p,N_j} = R_{p,N_j}^{'}(\rho) = -4 \frac{N_j-p+1}{u_j} - 2  + 2 \sqrt{R_{p-1,N_j}(\rho)} +  \frac{w_{p-1,N_j}}{2 u_j \sqrt{R_{p-1,N_j}(\rho)}}. 
\]
Now we are computing $\gamma_j = w_{j,N_j}$,
{\em i.e.}, we are interested in the $w_{p,N_j}$ for $p \leq j$.
In this range, $\sqrt{R_{p-1,N_j}(\rho)}= \frac{u_{j-p+1}}{u_j}$ by Proposition~\ref{prop:radicands-at-rho}, which gives
\[
w_{p,N_j} = -4 \frac{N_j-p+1}{u_j} - 2  + 2 \frac{u_{j-p+1}}{u_j} +  \frac{w_{p-1,N_j}}{2
u_{j-p+1}} 
= \delta_{p,j}  +  \epsilon_{p,j} \; w_{p-1,N_j},
\]
and it is then an easy exercise to obtain the explicit form of $w_{p,N_j}$.
\end{proof}

Set
\[
E_{s,p,j} = \prod_{r=s+1}^p \epsilon_{r,j} = \frac{1}{2^{p-s} \; \prod_{\ell = j-p+1}^{j-s} u_\ell}.
\]
Then 
\[
w_{p,N_j} = \sum_{s=1}^p \delta_{s,j} \; E_{s,p,j}
\]
and we can now turn to $\gamma_j = - w_{j,N_j}$: We write
\begin{align*}
\gamma_j 
=&
-\sum_{s=1}^{j} \delta_{s,j} \; E_{s,j,j}
\\ =&
\sum_{s=1}^{j}  \left( 4\frac{N_j-s+1}{u_j} +2 - 2 \frac{u_{j-s+1}}{u_j}\right) \; \; E_{s,j,j}
\\ =&\left( 4 \frac{N_j+1}{u_j} +2 \right) \sum_{s=1}^{j}  E_{s,j,j}
- 4  \sum_{s=1}^{j}  s E_{s,j,j}
- \frac{2}{u_j}   \sum_{s=1}^{j} u_{j-s+1}  E_{s,j,j}.
\end{align*}
and consider each term in turn.

\begin{lemma}
The sums $\sum_{s=1}^{j}  E_{s,j,j}$,  $\sum_{s=1}^{j} s  E_{s,j,j}$ and $\sum_{s=1}^{j} \frac{u_{j-s+1}}{u_j}  E_{s,j,j}$ all have a finite limit when $j \rightarrow \infty$.
\end{lemma}
\begin{proof}
It suffices to write, e.g., the first sum as 
$\sum_{s=1}^{j}  \frac{1}{2^{j-s} \; \prod_{\ell = 1}^{j-s} u_\ell}$ and to remember the
exponential growth of the sequence $(u_i)_{i\ge 0}$. 
The same argument holds for the second sum.
Finally, since $u_{j-s+1} < u_j$, the first sum is an upper bound of the last sum.
\end{proof}

This shows that
\[
\gamma_j \sim 4 \frac{N_j}{u_j} \sum_{s\geq 1}  E_{s,j,j}
\]
when $j \rightarrow \infty$. 
The relation $N_j = u^2_j - u_j + j$ then gives readily the following lemma, where the value of the constant has been computed numerically.
\begin{lemma}
\label{lemma:limite-gamma-j}
The term $\frac{\gamma_j}{u_j}$ has a finite, nonzero limit when $j\rightarrow  \infty$:
\[
\lim_{j\to\infty}\frac{\gamma_j}{u_j} = C \approx 6.347269145\dots
\]
\end{lemma}

\subsubsection{Asymptotic expansion of \protect$\prod_{i=1}^M {\lambda_i}$.}

\begin{lemma}
\label{lemma:produit-lambdas}
For $M\to\infty$ we have  
\[
\prod_{i=1}^M {\lambda_i} =
\sqrt{2\pi} \, \omega \sqrt{M} \, \left(\frac{M}{e}\right)^M  e^{2\sqrt{M} + \zeta(\frac{1}{2})} \; \left(1+ O\left(\frac{1}{\sqrt{M}}\right)\right),
\]
for some computable constant $\omega$ which is numerically $\omega\approx 0.1903\dots$
\end{lemma}

\begin{proof}
From the expression $\lambda_n = n + \sqrt{\lambda_{n-1}}$ and by bootstrapping, we obtain an asymptotic expansion for $\lambda_n$ when $n \rightarrow + \infty$:
\[\lambda_n = n + \sqrt{n} + \frac{1}{2} - \frac{3}{8\sqrt{n}} - \frac{1}{4n} + O \left( \frac{1}{n \sqrt{n}} \right) ,
\]
which gives
$\lambda_n = \left( n + \sqrt{n} + \frac{1}{2} \right) (1+ \omega_n)$ 
where $
\omega_n$ has order $n^{-\frac{3}{2}}$.
Consider now the product $\prod_{n=1}^M \lambda_n$ for $M$ large -- we shall take $M = N_j -j$ later on.
We can write it as $\prod_{n=1}^M ( n + \sqrt{n} + \frac{1}{2} ) \times \prod_{n=1}^M (1+ \omega_n)$,  and we consider separately each of the products.

\begin{itemize}
\item
We first concentrate on the product of the terms $1+\omega_n$.
We know that it has a finite limit $\omega$ if the series $\sum_n w_n$ is convergent, which is
indeed the case. This limit can therefore be computed as $\lim_{M\to\infty}\prod_{1 \leq n \leq M}
\frac{\lambda_n}{n+\sqrt{n}+\frac{1}{2}}$. The convergence, however, is slow (of order
$\frac{1}{\sqrt{M}}$). Thus the best we could achieve by numerical studies is 
$\omega \approx 0.1903\ldots$.

\item
We now turn to the product $\prod_{n=1}^M ( n + \sqrt{n} + \frac{1}{2} )$, which gives the asymptotic behaviour. We begin by writing it as
\[
M! \; \prod_{n=1}^M \left( 1 + \frac{1}{\sqrt{n}} + \frac{1}{2n} \right)
= M!\; \exp \left(\sum_{n=1}^M \log \left( 1 + \frac{1}{\sqrt{n}} + \frac{1}{2n}\right) \right).
\]
Now 
\begin{align*}
\sum_{n=1}^M \log \left( 1 + \frac{1}{\sqrt{n}} + \frac{1}{2n} \right) 
=&
\sum_{n=1}^M \left( \frac{1}{\sqrt{n}} + O \left( \frac{1}{n\sqrt{n}} \right) \right)
\end{align*}
where we can get effective bounds for the error terms. Observe that $\sum_{n=1}^M O \left(
\frac{1}{n\sqrt{n}} \right)= O\left(\frac1{\sqrt{M}}\right)$.
It remains to compute $\sum_{n=1}^M \frac{1}{\sqrt{n}}$, which is equal to $ 2 \sqrt{M} + \zeta(\frac{1}{2}) + O (\frac{1}{\sqrt{M}})$.
We finally obtain
\[
\prod_{n=1}^M \left( 1 + \frac{1}{\sqrt{n}} + \frac{1}{2n} \right) =
e^{2\sqrt{M} + \zeta(\frac{1}{2})} \; \left(1+ O\left(\frac{1}{\sqrt{M}}\right)\right) 
\]
\end{itemize}
and the final result by Stirling's formula. 
\end{proof}

By setting $M = N_j-j=u_j^2-u_j$ in Lemma~\ref{lemma:produit-lambdas}, we obtain
\begin{equation} 
\label{lemma:produit-lambdas-bis}
\prod_{i=1}^{N_j-j} \lambda_i = 
e^{\zeta(\frac{1}{2})-\frac{1}{2}}  \sqrt{2\pi} \, \omega \cdot
u_j^{2 u_j^2-2u_j+1} \; e^{-u_j^2+2 u_j} \; \left(1+O\left(\frac{1}{u_j}\right)\right).
\end{equation} 

\subsubsection{Putting all together}

We now substitute $C u_j$ for $\gamma_j$ in Eq.~\eqref{hNj}, according to
Lemma~\ref{lemma:limite-gamma-j}, and also plug in the asymptotic equivalent for the product
$\prod_{i=1}^{N_j-j} \lambda_i$ that comes from \eqref{lemma:produit-lambdas-bis}, to obtain
\eqref{hNjsecondexpression} and \eqref{constantD} 
which finishes the proof of Proposition~\ref{prop:decrease-h_Nj}.

\section{Bounded unary height vs. bounded unary length of bindings}
\label{comparison}

In Table~\ref{z:tab} we give numerical results of the constant and exponential terms for the
number of lambda-terms of bounded unary height and the number of terms where all bindings have 
bounded unary length.
We can see that the exponential terms for growing $k$ are quite similar in both cases.
Note that in case II the unary height is not bounded. Thus one might expect that bounding the
unary height is a much stronger restriction and that therefore the exponential growth rates should
exhibit a larger difference than they actually do. However, there is still a difference in the
{\em exponential} growth rates, which makes it appear reasonable. The quotient of the exponential
growth rates seems to tend to one which is as expected. 
\begin{table}[!h]
\begin{tabular}{>{$} l <{$} | >{\centering$} m{1.1in} <{$\centering} | >{\centering$} m{1in}
<{$\centering} | >{\centering$} m{1.2in} <{$\centering} | >{\centering$} m{1in} <{$\centering} }
& \multicolumn{2}{>{\centering} m{2.1in} <{\centering}|}{\textbf{Case I: Bounded unary height}}	&
	\multicolumn{2}{m{2.2in}}{\textbf{Case II: Bounded unary length of bindings}} \tabularnewline \cline{2-5}
\textbf{k} & \textbf{constant term}	& \textbf{exp. term}	& \textbf{constant term}	& \textbf{exp. term} \tabularnewline \hline \hline
\textbf{1} & 0.242613 & 2 & 0.21851 & 3 \tabularnewline
2 & 0.520859 & 2.90867 & 0.0866674 & 3.82843 \tabularnewline
3 & 0.231818 & 3.62279 & 0.0245664 & 4.4641 \tabularnewline
4 & 0.0838137 & 4.21545 & 0.00577152 & 5 \tabularnewline
5 & 0.0265937 & 4.73046 & 0.0011921 & 5.47214 \tabularnewline
6 & 0.0079582 & 5.19117 & 0.000223117 & 5.89898 \tabularnewline
7 & 0.0025262 & 5.61139 & 0.0000385385 & 6.2915 \tabularnewline
\textbf{8} & 9.31889 \times 10^{-5} & 6 & 6.21966 \times 10^{-6} & 6.65685 \tabularnewline
9 & 1.56532 \times 10^{-4} & 6.36386 & 9.46315 \times 10^{-7}& 7 \tabularnewline
10 & 1.99134 \times 10^{-5} & 6.70758 & 1.36666 \times 10^{-7} & 7.32456 \tabularnewline
\vdots & \vdots & \vdots & \vdots & \vdots \tabularnewline
133 & 2.16482 \times 10^{-152} & 23.8258 & 2.55075 \times 10^{-157} & 24.0651 \tabularnewline
134 & 1.30921 \times 10^{-153} & 23.9131 & 1.06018 \times 10^{-158} & 24.1517 \tabularnewline
\textbf{135} & 8.56995 \times 10^{-157} & 24 & 4.3907 \times 10^{-160} & 24.2379
\end{tabular}
\caption{Comparison of the constant terms and exponential terms values for the bounded unary height lambda-terms and bounded unary length of abstractions pointers lambda-terms.}\label{z:tab}
\end{table}

The constant factors differ significantly in both cases, but still they share a common behaviour:
They tend quite quickly to $0$ as $k \to \infty$.
One can also observe that for lambda-terms with bounded unary height in the cases where $k = N_j$
not only the term $n^{-\frac54}$ appears (instead of $n^{-\frac32}$), but also the constant factor 
behaves in a little different way: It is indeed smaller than one could expect. So far, we have no
explanation for this behaviour.

\section{Random generation and experiments}
\label{sec:observations}

\subsection{Random generation of lambda-terms}

To get a feeling of the ``average'' behaviour of a combinatorial object, a method of choice is the random generation of terms of large size.
We considered two methods to try to generate a random lambda-term of bounded unary height: the recursive method~\cite{FlZiVa94} and Boltzmann sampling.
Boltzmann samplers are powerful tools to generate objects in specified combinatorial classes
uniformly at 
random. They were introduced in \cite{DuFlLoSc04} and extended furthermore by numerous authors (see e.g. \cite{BoPo10,BFKV,FlaFuPi07,RoSo09}).
Note that, theoretically, a Boltzmann sampler can generate a tree of size close to $n$ on average
in linear time.

\begin{figure}[!htbp]
\centering
\includegraphics[height=5cm]{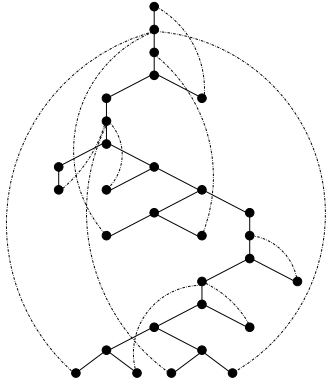}
\caption{A random lambda-term of size 30, with the edges from unary nodes to leaves.}
\label{lambda30}
\end{figure}

We considered Boltzmann sampling of a closed term, with different success depending on the unary
height: The efficiency decreases very quickly as the maximal  unary height grows. 
When $k=8$, we can generate terms of size $10000$ in a few seconds on a standard personal
computer. Figure~\ref{probaret} presents a term of size 6853 with unary height bounded by
8.\footnote{For large sizes and for the sake of  readability, we have not indicated the edges
between a unary node and the leaf labels.} 
\begin{figure}[!htbp]
\centering
         \includegraphics[height=6.5cm]{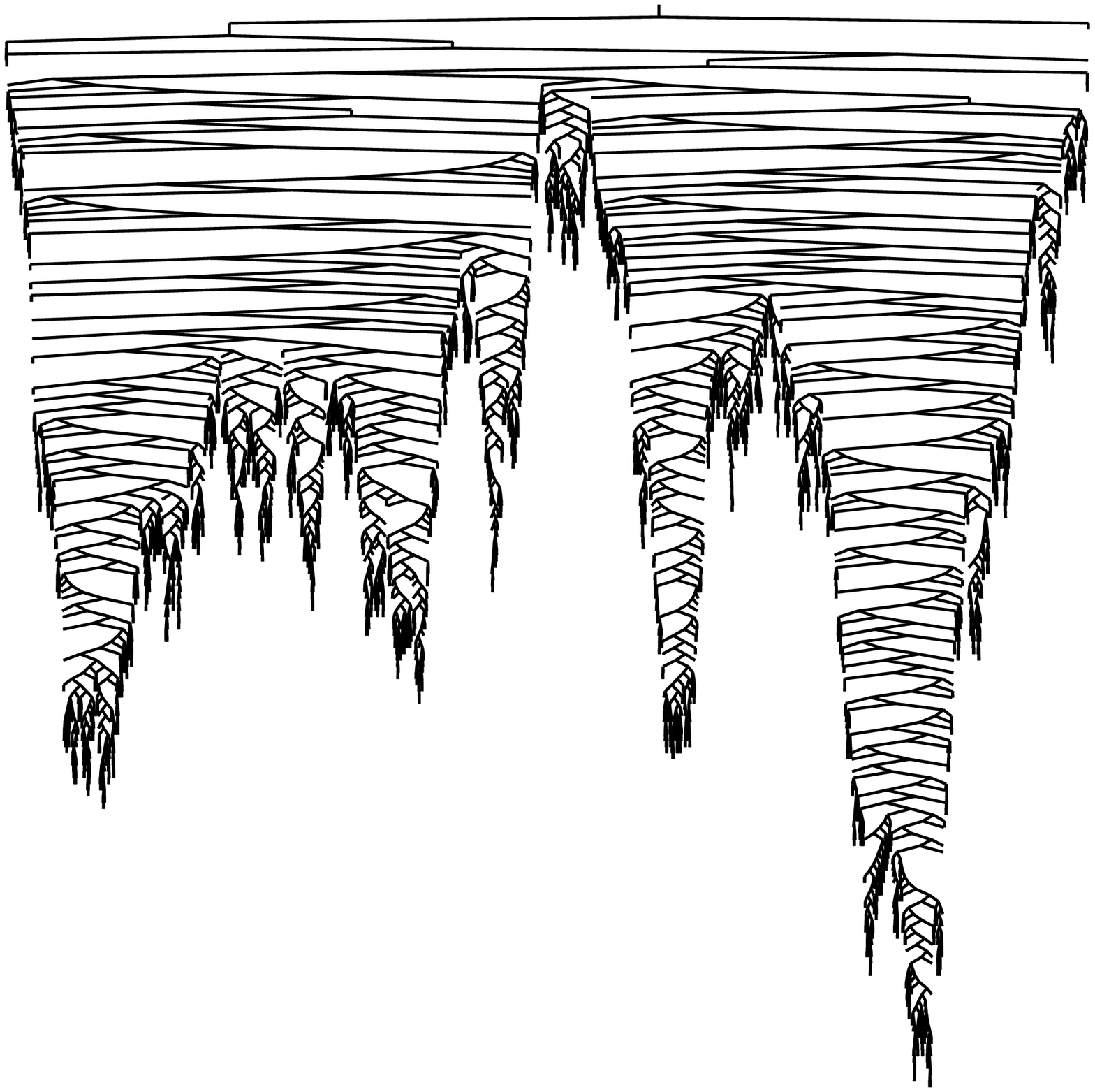} \hspace{7ex}
         \includegraphics[scale=0.35]{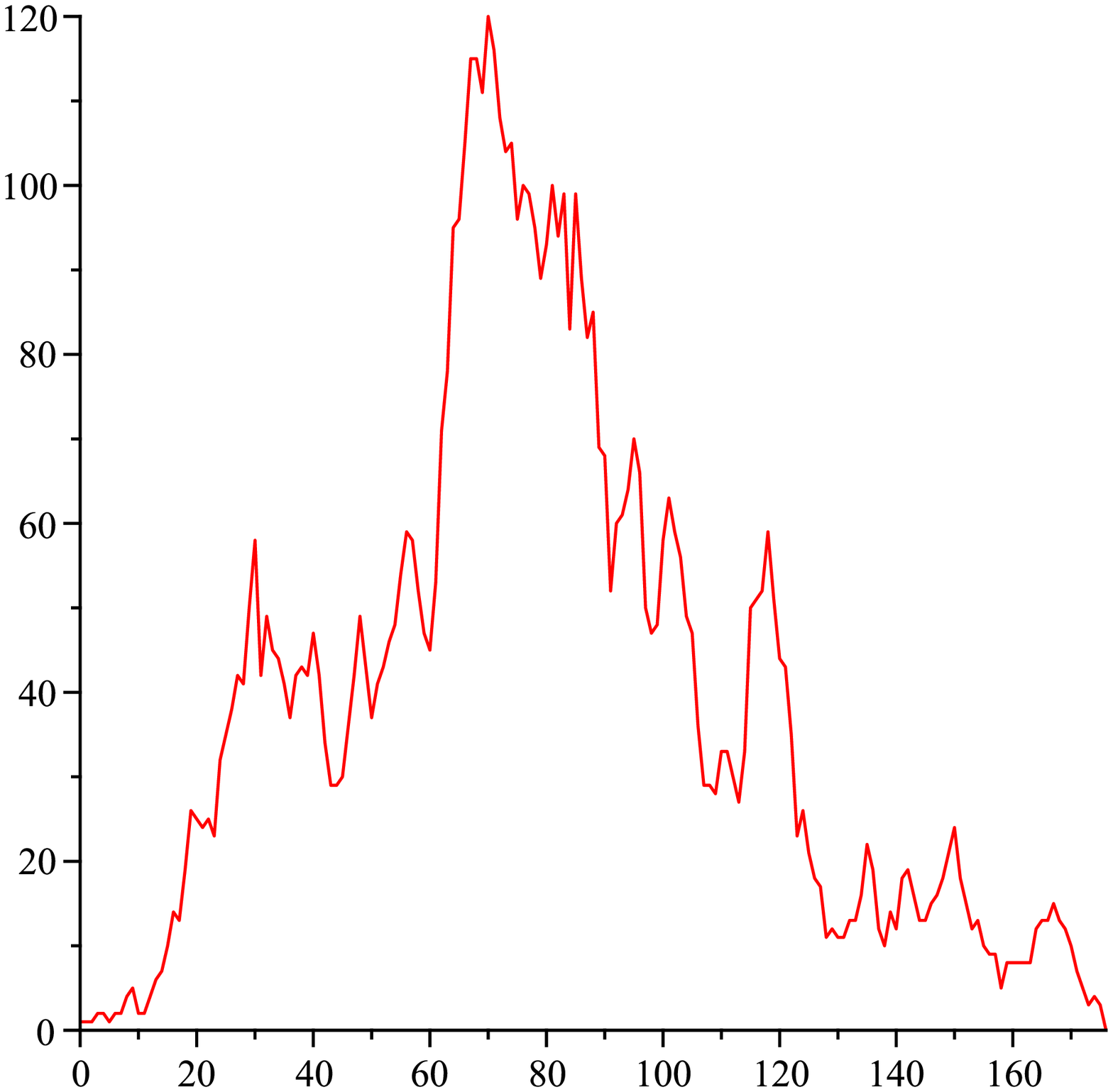}
	 \caption{The underlying Motzkin tree of a random lambda-term of size 6853 and unary height $\leq 8$ and its profile.}
\label{probaret}
\end{figure}

However, if we consider lambda terms with a maximal unary height of 135, a Boltzmann sampler is not able to produce objects of size larger than 200 in a ``reasonable'' time (less than one day). 
The explanation of the phenomenon is as follows: An ``average'' random lambda-term begins with a
large number of unary nodes; {\em cf.}~Figure~\ref{lambda200} (see also~\cite{DGKRTZ10} for a
result in
the same vein for a related model). Drawing the sufficient number of unary nodes has very low
\begin{figure}[!htbp]
\centering 
       \includegraphics[height=4.4cm]{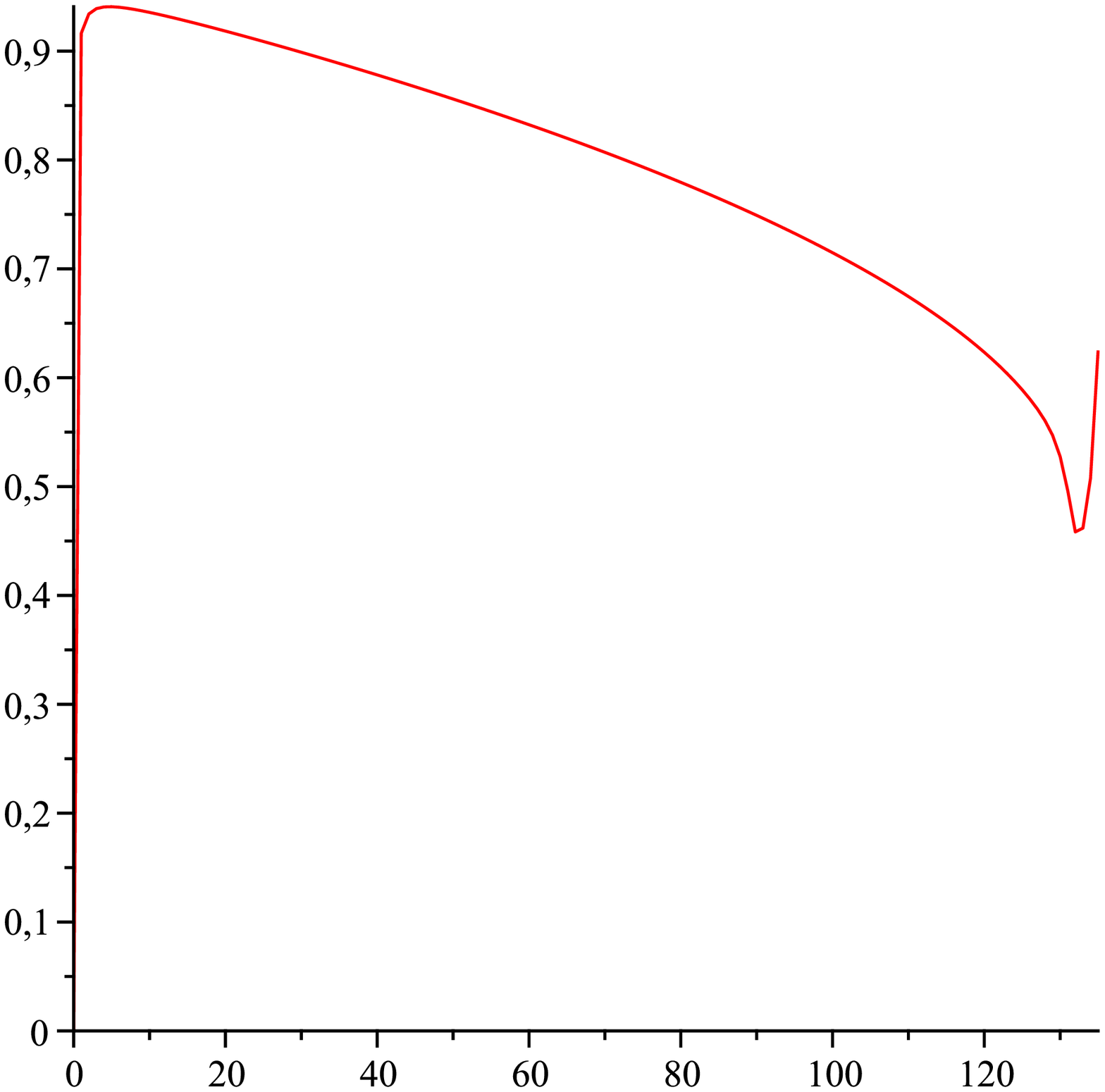}\hspace{5ex} 
       \includegraphics[scale=0.22]{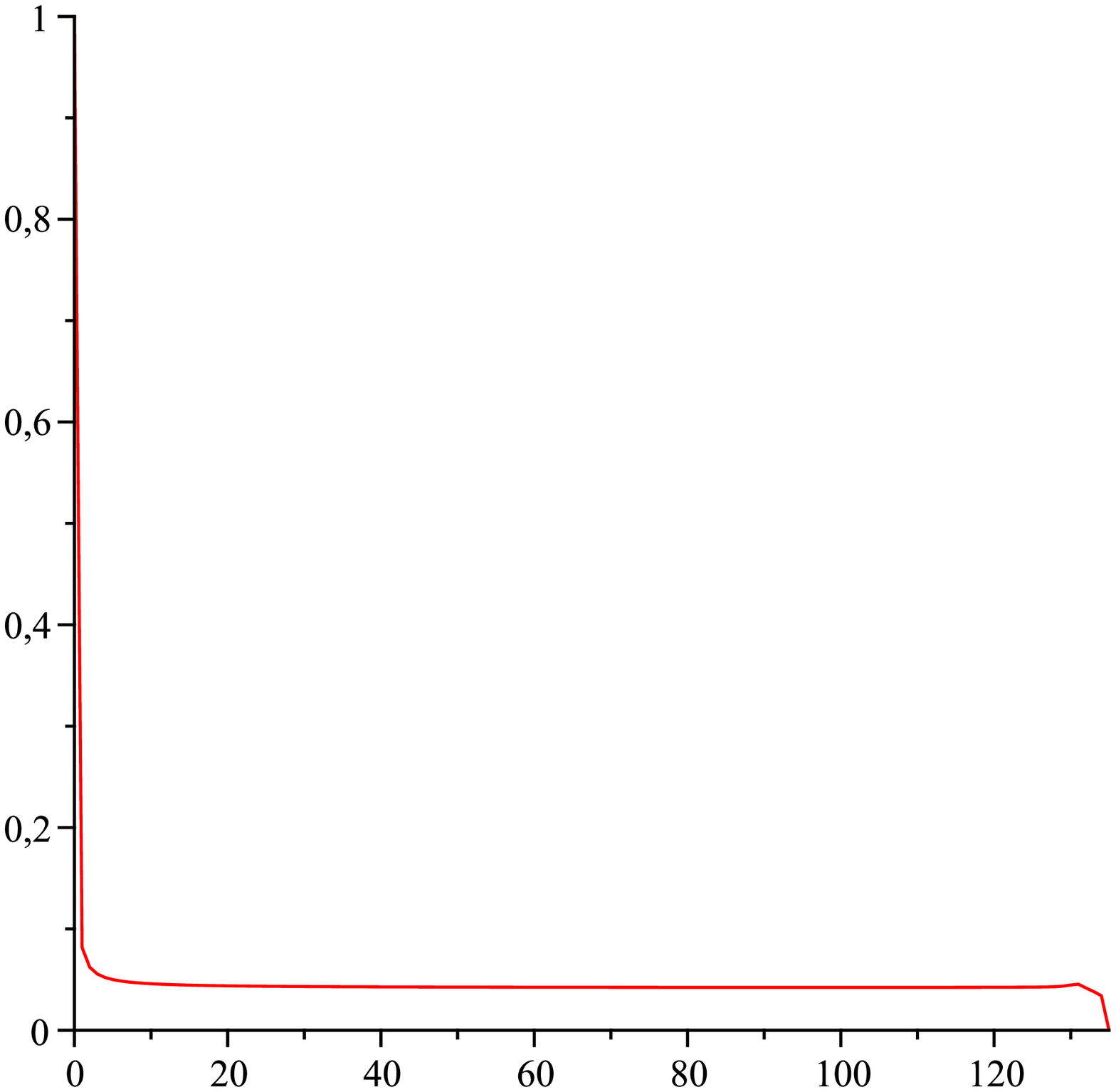}\hspace{5ex}
       \includegraphics[scale=0.22]{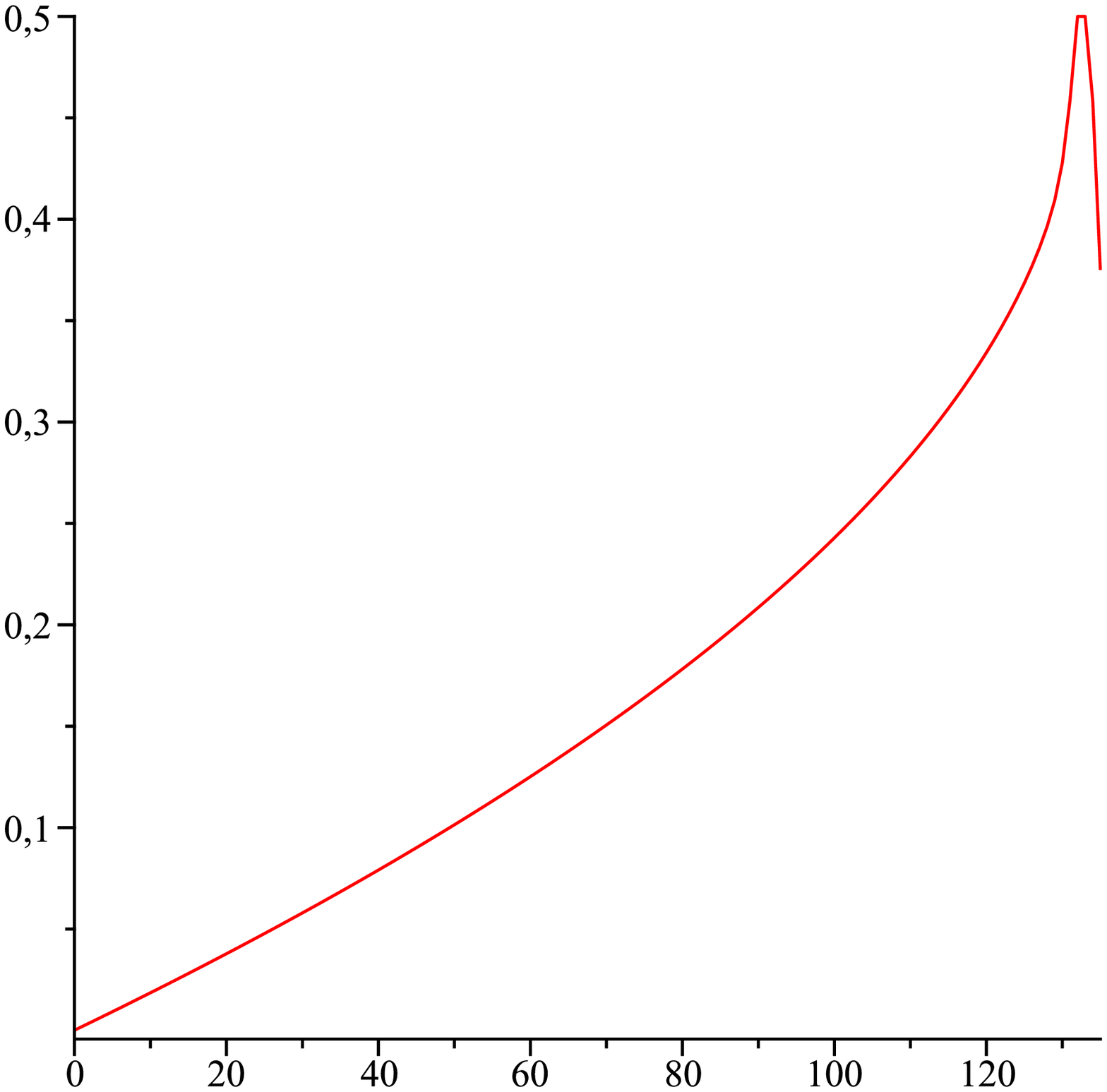}
\caption{Left: the probability that the singular Boltzmann sampler $\Gamma\mathcal{P}^{(k,135)}$
of objects in $\mathcal{P}^{(k,135)}$ stops immediately,. Middle: the probability that the sampler  $\Gamma\mathcal{P}^{(k,135)}$ calls $\Gamma\mathcal{P}^{(k-1,135)}$.
Right: the probability that the sampler $\Gamma\mathcal{P}^{(k,135)}$ independently calls 2 generators $\Gamma\mathcal{P}^{(k,135)}.$}
\label{arret}
\end{figure}
probability in the Boltzmann process.  Figure~\ref{arret} gives the various probabilities of
drawing a leaf, a unary node, or a binary node, plotted against the unary height 
(actually the
number of recursive calls to the generator, but the design of the generator is such that a call is
done if the unary height changes).
After a (long!) starting phase, where the probability of stopping is larger than 0.9, 
the Boltzmann sampler becomes efficient.
In other words, Boltzmann sampling is linear, but with a constant depending on the maximum unary
height which grows {\em very} quickly: The recursive form of the specification of lambda-terms and
their varying behaviour makes them not well amenable to random generation with a Boltzmann sampler. 

We have thus turned to the recursive method. Using the Maple package Combstruct, we have been able
to generate quickly enough lambda-terms of size 200 and unary height bounded by 200--which means
that there is \emph{de facto} no restriction on the unary height of the lambda-term.
Figure~\ref{lambda200} shows what can be considered as a ``generic'' lambda-term for this size.

\begin{figure}[!htbp]
\centering
       \includegraphics[scale=0.24]{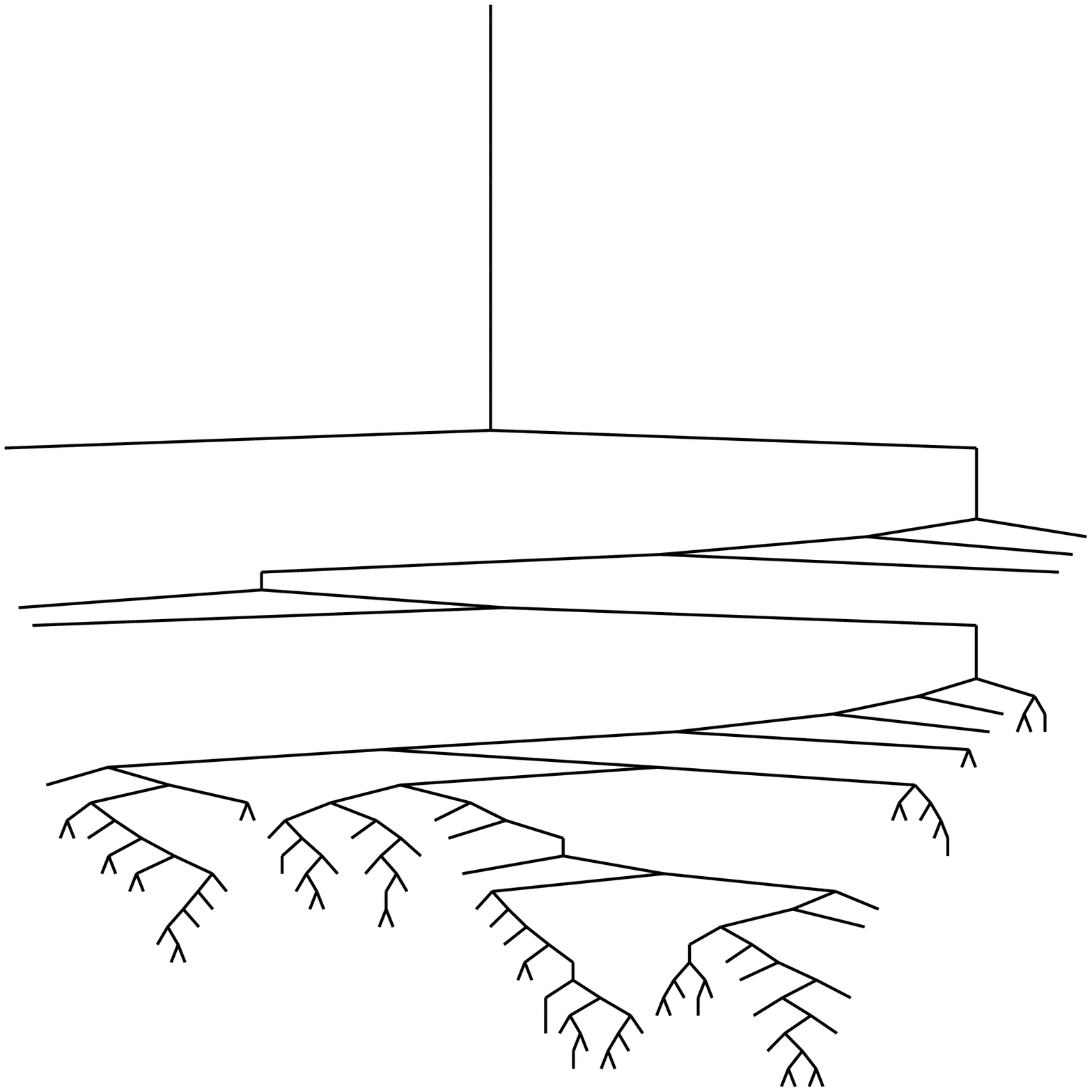}\hspace{2ex}
       \includegraphics[scale=0.24]{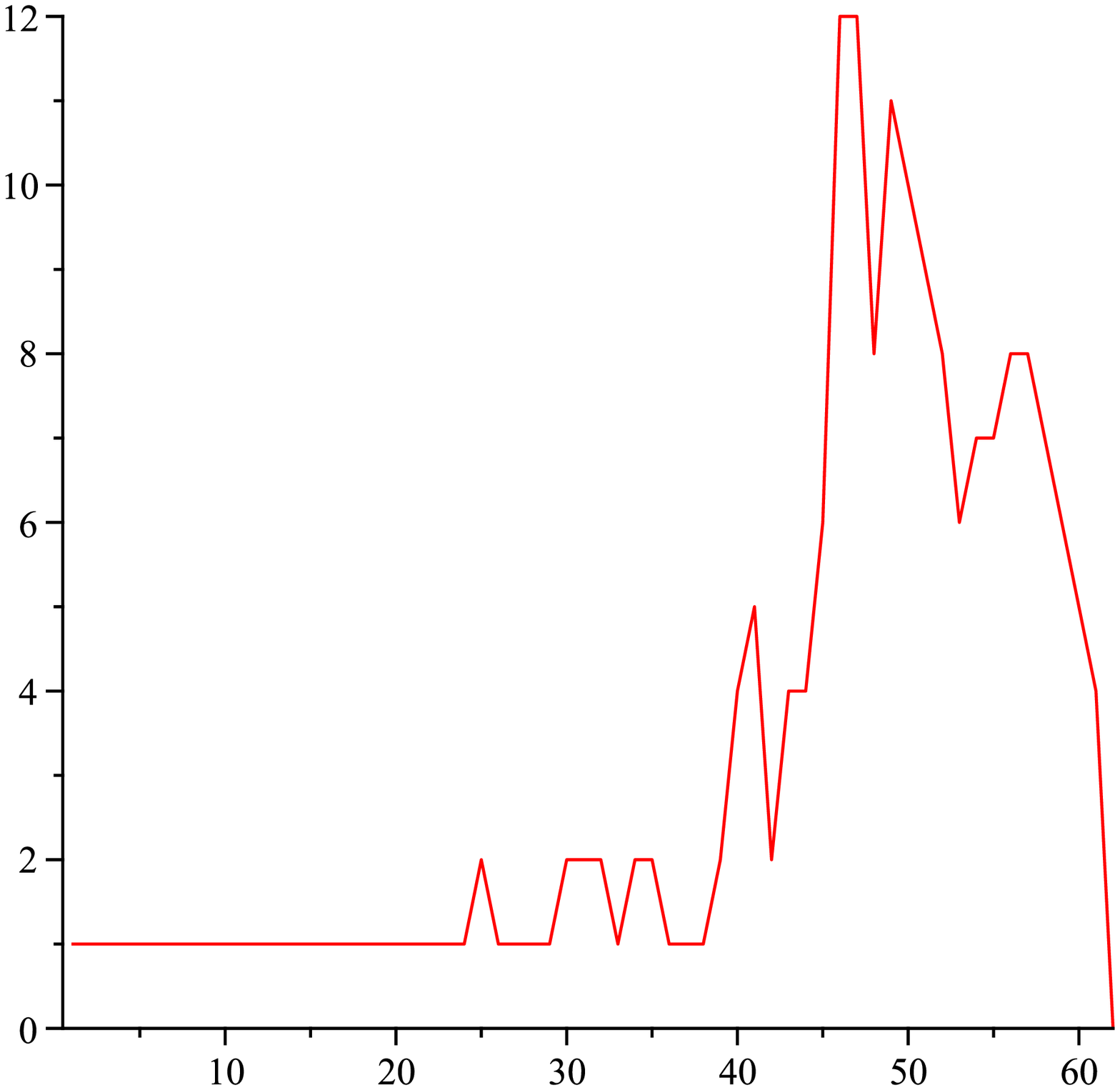}\hspace{2ex}
       \includegraphics[scale=0.24]{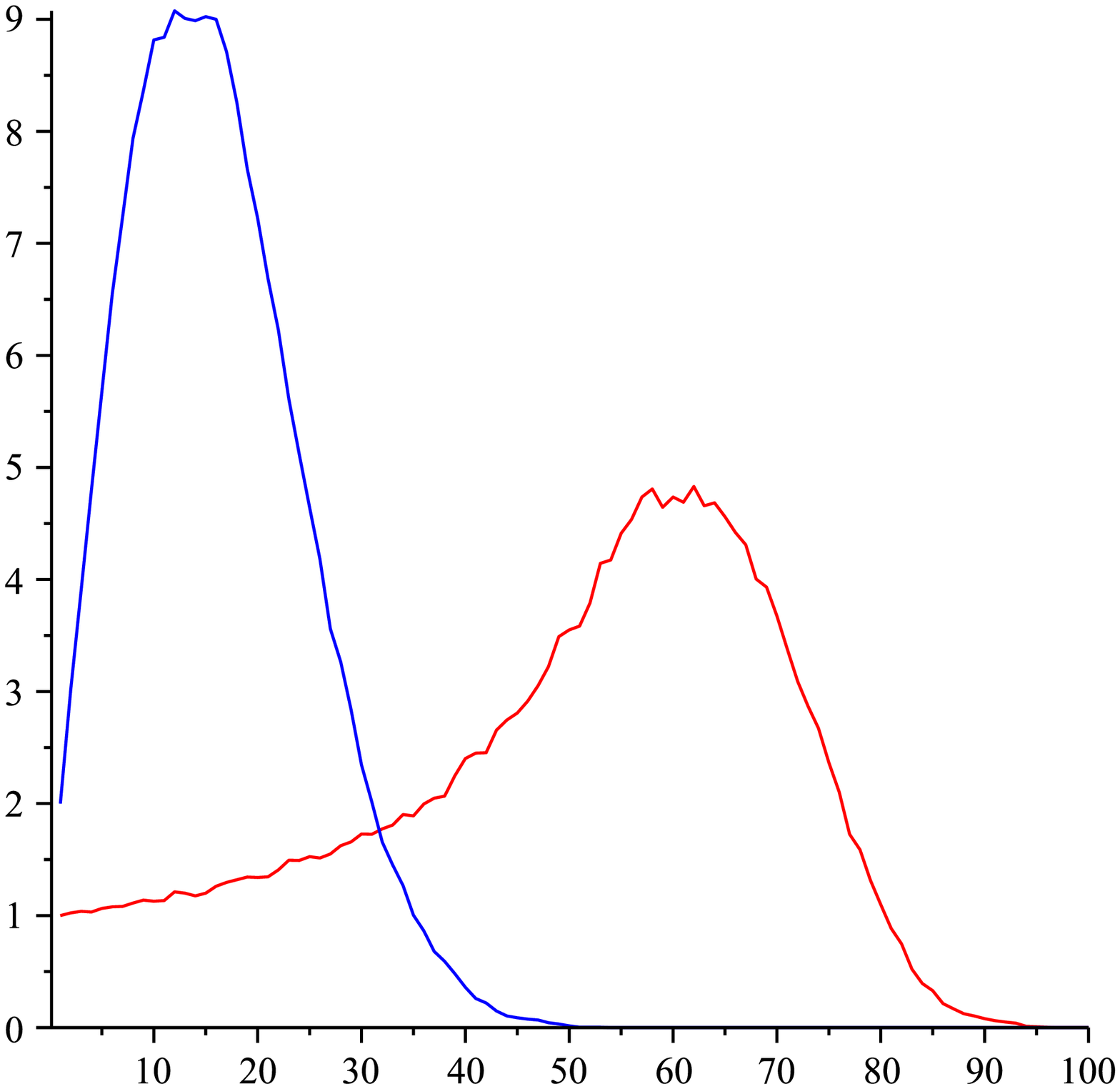}
\caption{Left: a lambda-term of size 200.
Middle: its profile. Right: the average profile (red) computed over 500 random lambda-terms,
compared with the average profile for plane binary trees (blue: the Airy function)}
\label{lambda200}
\end{figure}

Both classes, the one with bounded unary height and the one where all bindings have bounded unary
length, can be used to approximate generic lambda-terms. But unfortunately, also in the case of
bounded unary length of bindings we are facing the same difficulties when trying to generate them
with a Boltzmann sampler. The probabilities for generating leaves, unary and binary nodes looks
very similar to Figure~\ref{arret}. This fact can be explained as follows: For both classes of 
restricted lambda terms, the dominant radicand is either close or equal to the innermost radicand.
But the Boltzmann sampler generates these from outside inwards. That is meant in the following
sense: Each square-root is the analytical analogue of the lifting from one unary level to the next
one (\emph{cf.} \eqref{def-Pik} and \eqref{wherenestedrootscomefrom} in order to see this). The
Boltzmann sampler builds an object by starting from the root and attaching more and more nodes.
So, the head of the term, \emph{i.e.} the subtree comprising all nodes of unary height zero, is
precisely the object corresponding to the outermost root; and this is generated before the nodes
with larger unary height. But note that the generating function of the class of heads has a larger
dominant singularity. Hence the tuning parameter of the Boltzmann sampler is far away from this 
singularity, thus giving the sampler a strong bias towards stopping. On the other hand, moving the
parameter into an interval where the sampler works efficiently means that it is outside the domain
of analyticity of the generating function associated with lambda-terms. This implies that we have
a positive probability that the sampling process never stops. So the sampler becomes even more
inefficient than with the badly chosen tuning we had before moving it to the allegedly better
region. Bodini et al.~\cite{BLR15} developed a general framework for Boltzmann sampling for which
tuning parameters outside the region of convergence of the associated generating function can be
used. This relies on anticipated rejection and might help to improve the Boltzmann samplers for
generating random lambda-terms. 

For restricted Motkzin trees the situation is totally different, because the
dominant singularity comes from the outermost radicand. Thus the Boltzmann sampler starts to
generate the object by generating subobjects corresponding to the root which determines the
singularity, and we can choose the tuning parameter so that it in the optimal region.

\subsection{Shape of a typical lambda-term}

Being able to draw repeatedly random lambda-terms allows us to make tentative conjectures on their
various parameters: profile, height, {\em etc.}

\begin{figure}[!htbp]
\centering 
\includegraphics*[scale=0.78]{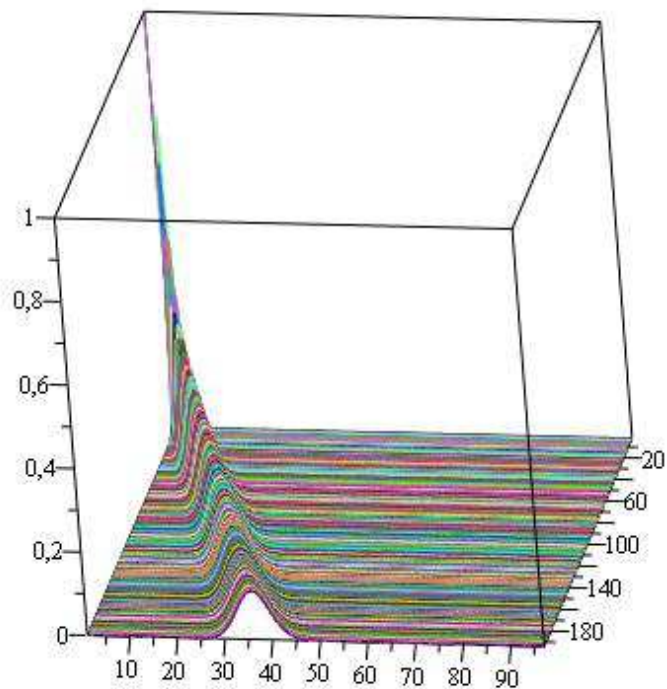}
%
%
\caption{Distribution of lambda-terms of size $n\in[1,\dots,198]$ and unary height $k\in[1,\dots,98]$}
\label{distun}
\end{figure}

We have plotted in Figure~\ref{distun} the ratio between the number of lambda-terms with
unary height exactly $k$ and size $n$, and the number of lambda-terms of size $n$ (without restriction on the height).
The figure suggests that, for any given size~$n$, the unary height is close to a Gaussian distribution. 
In particular,
this gives some experimental explanation to the change of difficulty which we encountered when
generating terms of small unary height (size about 10\,000, unary height bounded by 8) and terms
of fairly large unary height (size about 10\,000, unary height bounded by 135): The wave indicates 
the ``good'' estimate for the number of abstractions in a lambda-term; for instance, if we
consider lambda-terms of size 198, then the vast majority of these terms has a unary height
between 25 and 50. 

Figure~\ref{lambda200}  shows a generic lambda-term, its profile (number of nodes at each level)
and the profile averaged on 500 random lambda-terms, together with the average profile of a plane
binary tree, which is up to scaling identical with that of Motzkin trees since both tree classes
are simply generated. From our simulations we can make several empirical observations: 
\begin{itemize} 
\item The distribution of the profiles is poorly concentrated (this is also the case for plane binary trees). 
\item The levels containing the larger number of nodes are much farther from the root than in
plane binary trees. 
\item A simulation of the distribution for the total (unary) height  
also shows a clear difference to plane binary trees: The average (unary) height seems to grow
almost linearly (actually proportional to $n/\log n$), not proportional to $\sqrt{n}$ as the
height of binary trees or the unary height (and also the height) of Motzkin trees.
Accordingly, the width of lambda-terms appears to grow as $\log n$.
\item A random lambda-term usually begins with a large number of successive 
unary nodes interspersed with a few binary nodes; most binary nodes appear further down.
\end{itemize}

Figure~\ref{probaret} shows the underlying Motzkin tree of a large lambda term of bounded unary
height (the bound is 8) and its profile. Simulations for the case of bounded unary length of
binding lead to similar pictures. Certainly, one of the reasons is that the bound 8 is still very
small to exhibit a visible qualitative difference between the two classes of lambda-terms. On the
other hand, it is also possible that the shape of the underlying Motzkin tree is too similar in
both models if $k$ is relatively small. 

\begin{figure}[!htbp]
\centering 
         \includegraphics[height=5.2cm]{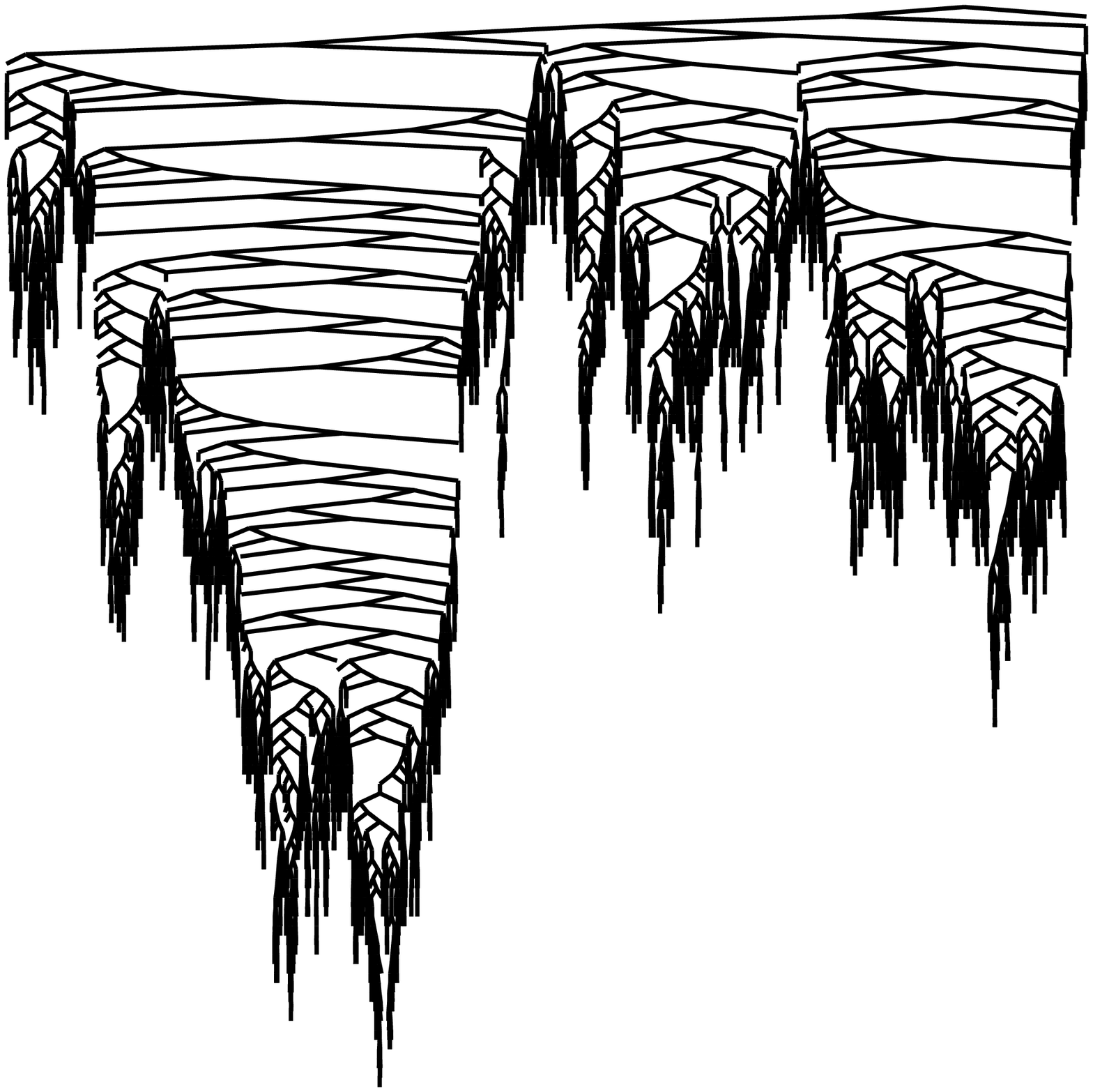} \hspace{7ex}
         \includegraphics[scale=0.28]{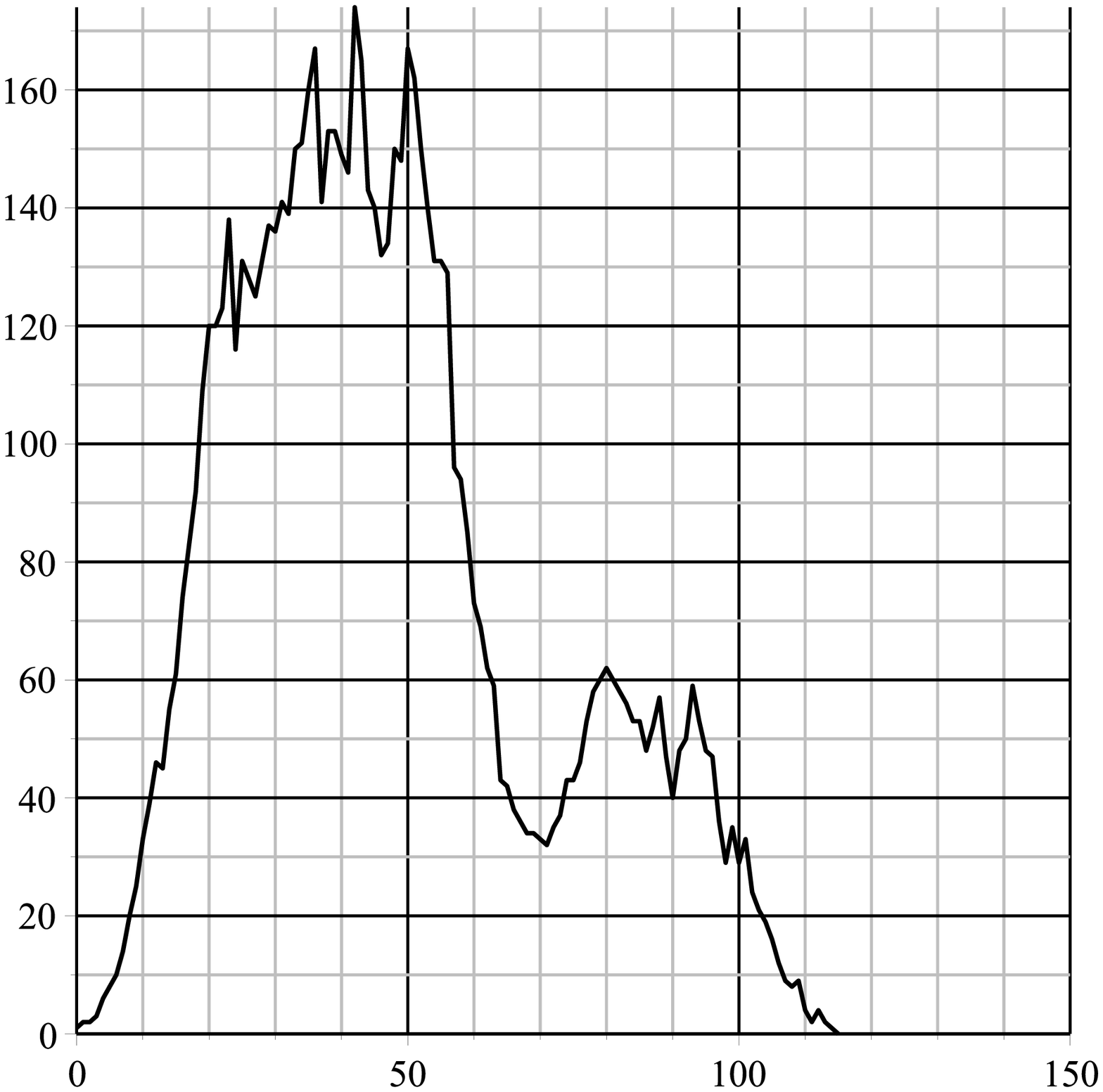}
         \caption{A random Motzkin tree of size 8368 and unary height $\leq 8$ and
its profile.} 
\label{randmotz}
\end{figure}

\begin{figure}[!htbp]
\centering 
         \includegraphics[height=5.2cm]{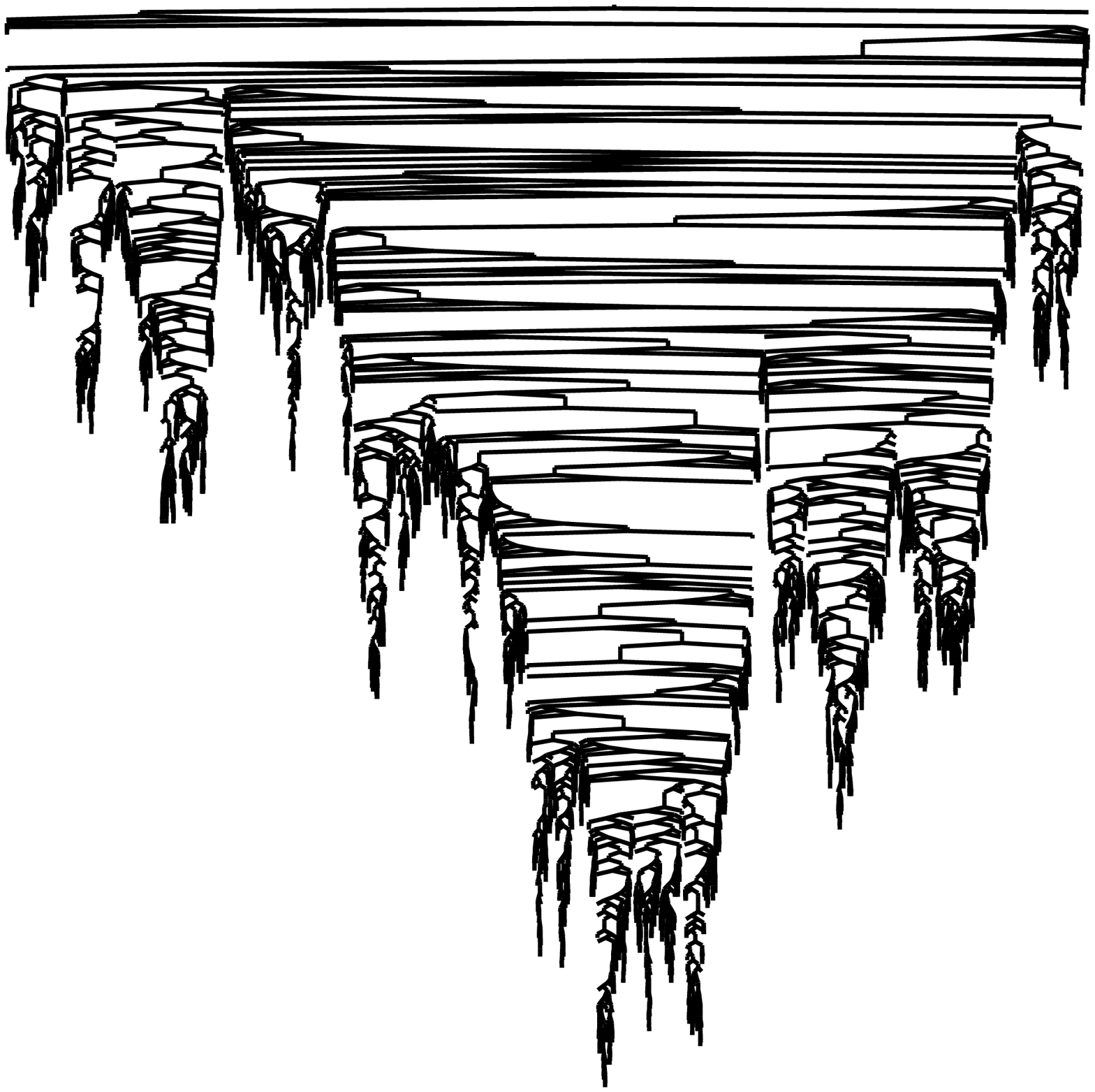} \hspace{5ex}
         \includegraphics[scale=0.28]{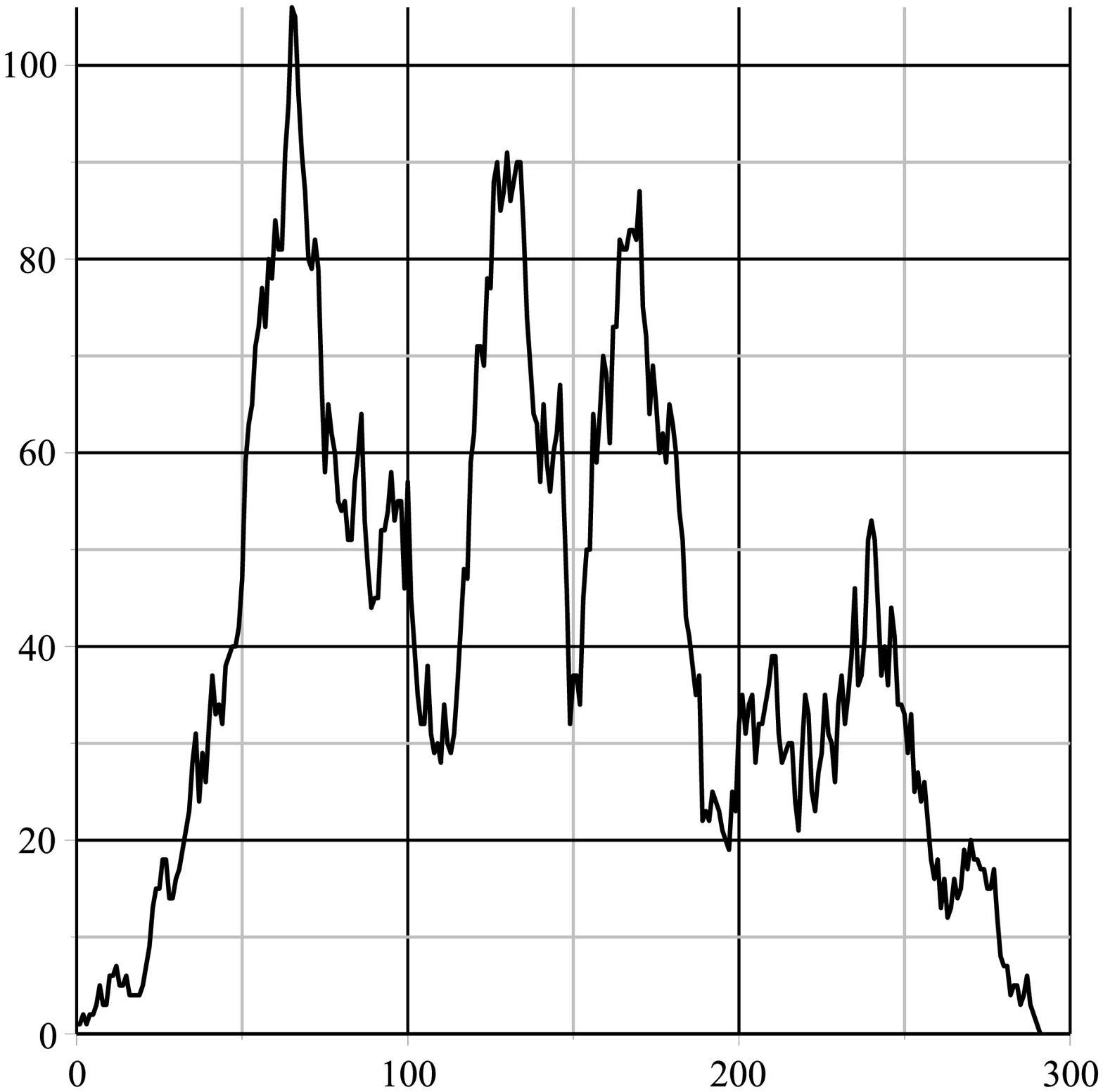}
	 \caption{A random Motzkin tree of size 11995 and unary height $\leq 100$ and its
profile.}
\end{figure}

In contrast to this is the behaviour of Motzkin trees of bounded unary height. As
Figure~\ref{randmotz} shows, the position of most of the unary nodes seems to be near the leaves
whereas in a random lambda-term of bounded unary height most the unary nodes form a string
starting at the root (\emph{cf.} Figure~\ref{lambda200}).

\section{Conclusion and perspectives}
\label{sec:conclusion}

In this paper we have studied several classes of lambda-terms; see also \cite{BGGJ13} for further
classes. 
It is clear that allowing pointers from unary nodes to leaves is the main factor that determines the complexity of such structures, and the more unary nodes we allow, the farther we are from trees.
Indeed, allowing pointers from internal nodes to leaves amounts to leaving the realm of trees for
that of directed acyclic graphs.
As regards the enumeration of restricted classes, bounding the number of unary nodes as well as
bounding the unary length of bindings (which is locally bounding the number of levels of nesting
for abstractions) leads to an 
asymptotic behaviour that resembles that of trees (of type $n^{-3/2} \rho^n$), even though the
latter is already of considerable combinatorial complexity. In contrast, bounding
the unary height (which means globally bounding the number of levels of nesting for abstractions) 
exhibits an unusual behaviour. 

Among other facts, we have discovered the unexpected behaviour of the
position of the dominant radicand, which jumps according to some function behaving as
$\log (\log (k))$, with $k$ being the bound for the unary height of a lambda-term.
Theorem~\ref{theo-4.1} characterizes precisely these jumps and the
asymptotic number of lambda-terms with bounded unary height.
The enumerative result looks tree-like unless the bound for the
unary height belongs to the special sequence $(N_j)_{j\ge 1}$. 

The fact that the generating function has a nested square-root representation, but the
position of the dominant radicand depends on the specific restriction appeared also in our
studies of Motzkin. We studied them for comparison reasons since they form the underlying
structure of lambda-terms. Regarding the asymptotic enumeration results, they exhibit the
tree-like pattern, except if we impose a very unnatural shape: In the case where all leaves have
to be at the same unary height we observe a different singularity of different type. In this case
all radicands are dominant, as opposed to the other cases where either the innermost or the
outermost radicand is dominant. 

In contrast to this stands the behaviour of lambda terms of bounded
unary height where the position of the dominant radicand is depending on the bound for the unary
height. This phenomenon requires further explanation. 

Further investgations indicate that the strange jumps in the behaviour are related to the
distribution of the leaves in a lambda-term with bounded unary height. It seems that they are
concentrated in the last few levels (level counting w.r.t. unary height) while the lower levels
contain almost no leaves. The number of these levels seems to be doubly logarithmic in the size of
the terms and whenever $k=N_j$ for some $j$, then a new level ``enters'', meaning that it contains
then a significant number of leaves. So, the special values of $k$ are those where a transition
takes place from $\ell$ to $\ell+1$ levels, filled with almost all the leaves of the lambda-term.
In lambda-terms belonging to the class where all bindings have bounded unary length we expect that 
the distribution of the different types of nodes (unary, binary and leaves) is more uniformly
distributed within their underlying Motzkin trees than in the bounded unary height case. This is
indicated by generation of small objects (size 100-200, \emph{cf.} Figure~\ref{smallterms}). 

\begin{figure}[!htbp]
\centering
         \includegraphics[height=4cm]{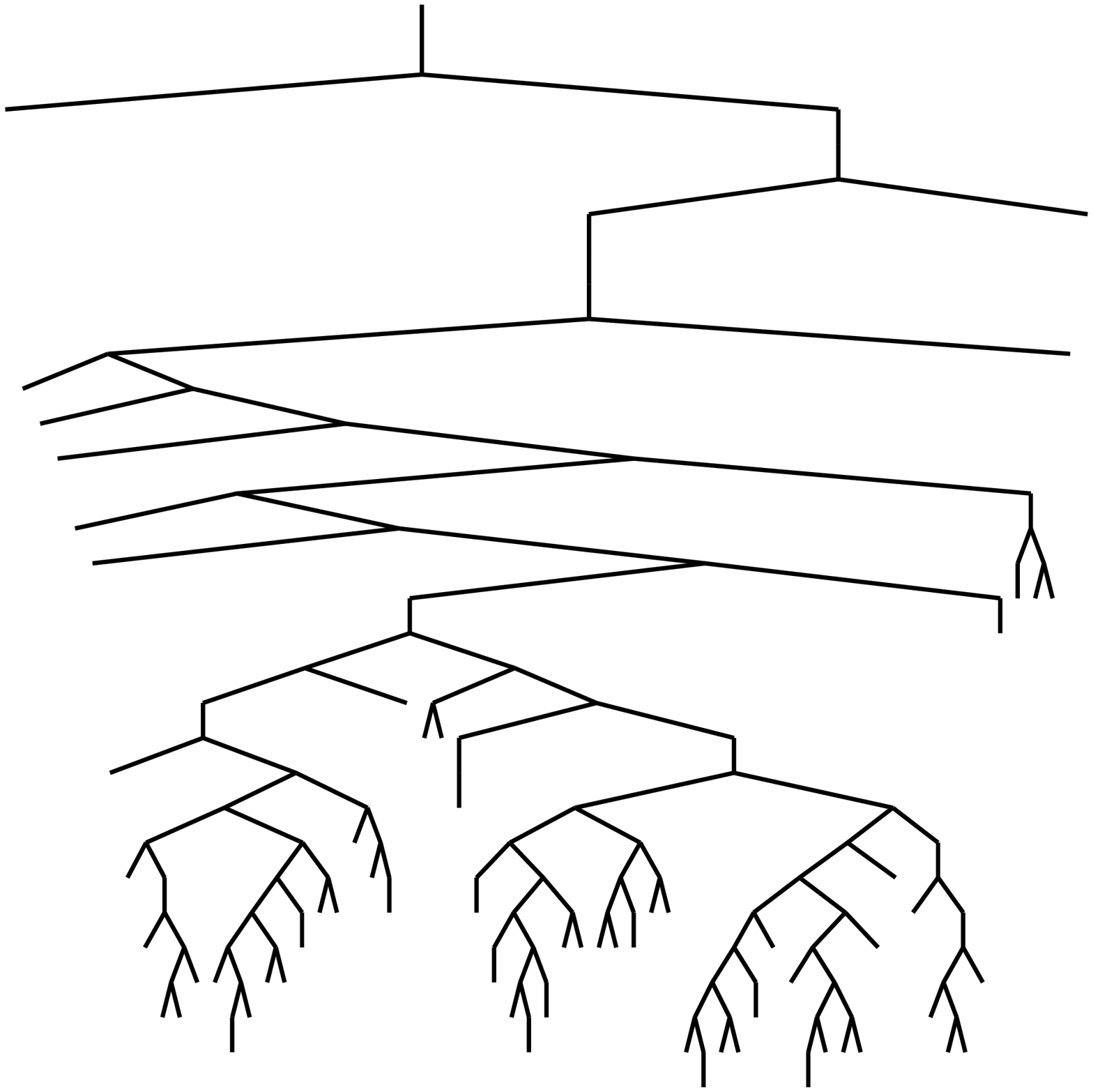} \hspace{7ex}
         \includegraphics[height=4cm]{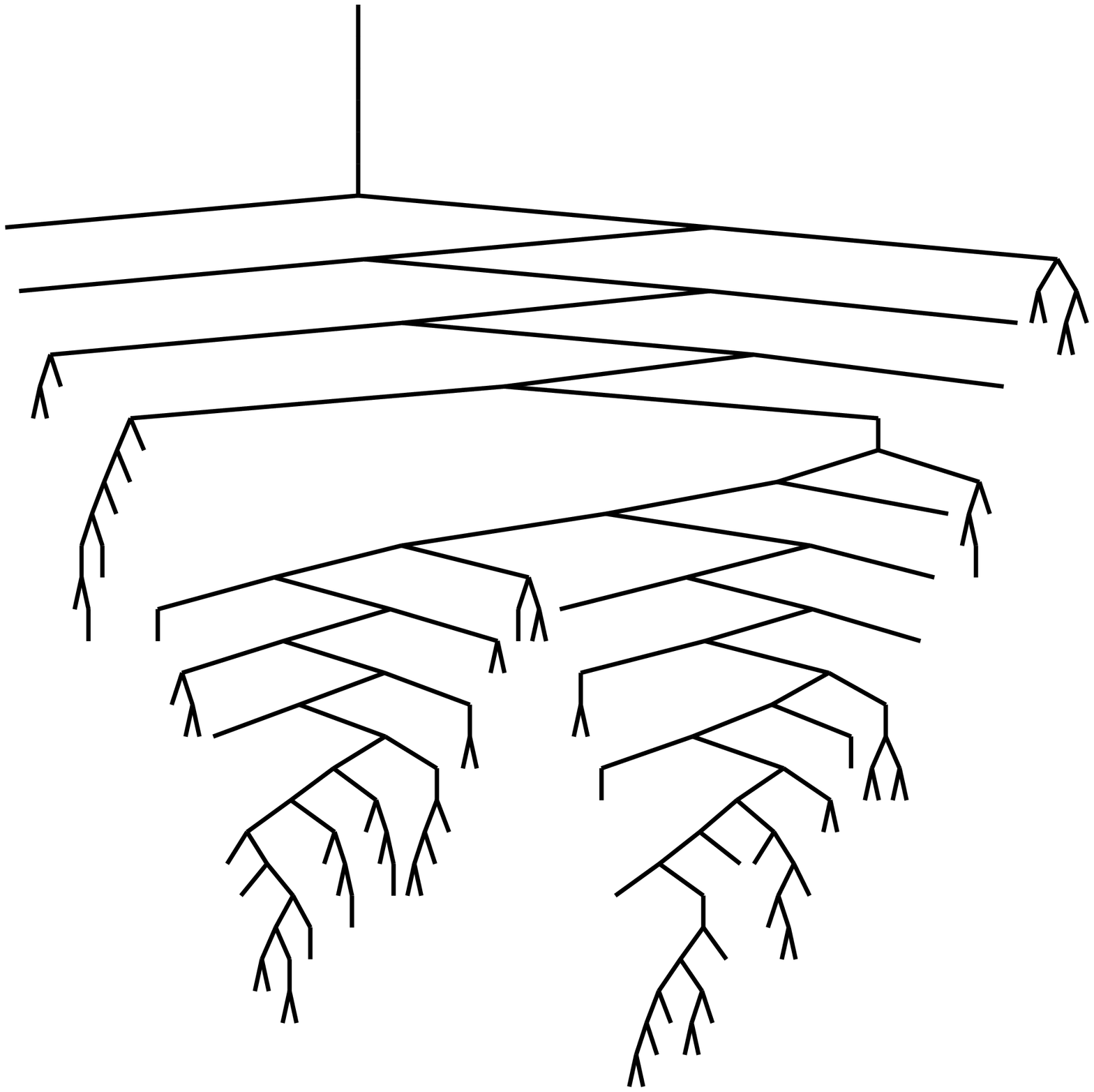} 
\caption{Left: an underlying Motzkin tree of a lambda-term with bounded unary height, Right: an
underlying Motzkin tree of a lambda-term bounded unary length of bindings.}
\label{smallterms}
\end{figure} 

A byproduct of our work concerns Boltzmann samplers: By trying to use them for the random
generation of lambda-terms, we have pushed them to their limit. It turned out that Boltzmann
samplers have serious difficulties to generate generic lambda-terms of large size. The same is
true if the unary height or the length of the bindings is bounded. From an analytic view
point, the reason is certainly that the dominant singularity does come from radicands lying in a
deep level of nestings. Another reason might be that the multiplicative constants in the
asymptotic main terms decrease so rapidly with growing $k$. We analyzed these constants
exhaustively for the case with bounds on the binding length, and for $k=N_j$ in the bounded unary
height case. It remains an open problem to carry out a precise analysis for all $k$ and to
understand those irregularities discussed in Section~\ref{comparison}. We feel that it might be
also possible to improve Boltzmann random generation, when we wish to apply it to combinatorial
structures for which Boltzmann samplers mostly produce either small or infinite objects. Recently,
a framework for Boltzmann sampling has been developed by Bodini et al.\cite{BLR15} which
generalizes the existing one in a direction which might help to overcome some of the difficulties
we are facing in the generation of lamda-terms.

We next mention that our approach can be extended to study formulas of quantified logic: Instead
of a single type of unary nodes, we have as many types as different quantifiers (usually two:
$\forall$ and $\exists$; $l$ in general). We also have as many types of binary nodes as there are binary connectors (e.g., two when we consider the connectors $\wedge$ and $\vee$; $h$ in general); here we have studied the case $l=h=1$.
We expect that allowing different types of unary nodes will introduce only a multiplicative
coefficient in our results, whereas allowing different types of binary nodes will change the
singularities and thus the exponential growth. 

Finally, in terms of average properties and growth, lambda-terms widely differ from the usual
models for trees such as simple families \cite{MM78} or increasing trees \cite{BeFlSa92}, 
for which we know the behaviour of
classical parameters: number of trees of given size, profile, \emph{etc.} 
Indeed they seem to behave,
in some sense, like ``ornamented'' paths, {\em i.e.} long strings onto which relatively 
small subterms are grafted. 

Of course, such results need to be explained and quantified more rigorously. Let us also mention
that the enumeration of (unrestricted) lambda-terms is still an open problem, which has to be 
solved in order to study such parameters as the (average) unary height, the profile, {\em etc.}

An interesting question is the probability that a random lambda-term is in normal form. We are currently studying this problem for restricted classes of lambda-terms and hope to give results in a forthcoming paper.

\medskip
\begin{ack}
The authors thank Pierre Lescanne for pointing out reference \cite{YCER11}. 
\end{ack}

\end{document}